\documentclass[11pt,a4paper,reqno]{article}
\usepackage{amsmath}
\usepackage{amsthm}
\usepackage{enumerate}
\usepackage{amssymb}

\usepackage{verbatim}
\usepackage{url}
\usepackage[latin1]{inputenc}

\usepackage{caption}
\usepackage{graphicx,color}
\usepackage[round, longnamesfirst]{natbib}

\def\N{{\mathbb N}}
\def\Z{{\mathbb Z}}

\def\R{{\mathbb R}}
\def\Eb{{\mathbb E}}
\def\Vb{{\mathbb V}}
\def\F{{\mathcal F}}
\def\T{{\mathcal T}}
\def\G{{\mathcal G}}
\def\Vc{{\mathcal V}}

\def\ebf{{\mathbf e}}

\def\tSk{\tau_{S_k}}

\def\Trnn{T(I,\hat{v}_{\rho_{\nu(n)}})}

\def\rnn{\rho_{\nu(n)}}
\def\mut{\mu_\tau}
\def\muS{\mu_S}
\def\raw{\rightarrow}

\def\mt{m_\tau}
\def\Mt{M_\tau}
\def\Tc{T_c}
\def\Td{T_{\text{delay}}}
\def\Nc{N_c}
\def\Pt{P_\tau}

\def\tp{\tau^{\prime}}

\def\Tp{T^{\prime}}
\def\np{n^{\prime}}

\def\tB{\tilde{B}}

\def\Ank{A_{n_k}}

\def\tlow{t^\prime}
\def\thi{t^{\prime\prime}}

\DeclareMathOperator{\E}{E}
\DeclareMathOperator{\Var}{Var}
\DeclareMathOperator{\Cov}{Cov}

\DeclareMathOperator{\dist}{dist}
\DeclareMathOperator{\ind}{{\bf 1}}

\def\be{\begin{equation}}
\def\ee{\end{equation}}
\def\bea{\begin{equation*}}
\def\eea{\end{equation*}}
\def\begs{\begin{split}}
\def\ends{\end{split}}

\newtheorem{thm}{Theorem}[section]
\newtheorem{lma}[thm]{Lemma}

\newtheorem{prop}[thm]{Proposition}
\newtheorem{claim}{Claim}

\newtheorem{df}[thm]{Definition}

\theoremstyle{remark}
\newtheorem{preremark}[thm]{Remark}
\newtheorem{preex}[thm]{Example}

\newenvironment{remark}{\begin{preremark}}{\qed\end{preremark}}

\theoremstyle{definition}
\newtheorem*{acknow}{Acknowledgements}
\newtheorem*{keywords}{Keywords}
\newtheorem*{msc}{MSC 2010}

\numberwithin{equation}{section}

\begin{document}

\title{Asymptotics of first-passage percolation on 1-dimensional graphs}

\date{\today}

\author{Daniel Ahlberg\thanks{Department of Mathematical Sciences, University of Gothenburg, and Department of Mathematical Sciences, Chalmers University of Technology.}}

\maketitle

\begin{abstract}
In this paper we consider first-passage percolation on certain 1-dimensional periodic graphs, such as the $\Z\times\{0,1,\ldots,K-1\}^{d-1}$ nearest neighbour graph for $d,K\geq1$. We find that both length and weight of minimal-weight paths present a typical 1-dimensional asymptotic behaviour. Apart from a strong law of large numbers, we derive a central limit theorem, a law of the iterated logarithm, and a Donsker theorem for these quantities. In addition, we prove that the mean and variance of the length and weight of the minimizing path between two points are monotone in the distance between the points.

The main idea used to deduce the mentioned properties is the exposure of a regenerative structure within the process. We describe this structure carefully and show how it can be used to obtain a detailed description of the process based on classical theory for i.i.d.\ sequences. In addition, we show how the regenerative idea can be used to couple two first-passage processes to eventually coincide. Using this coupling we derive a 0--1 law.

\begin{keywords}
First-passage percolation, renewal theory, classical limit theorems, 0--1 law.
\end{keywords}

\begin{msc}
60K35, 60K05.
\end{msc}

\end{abstract}

\section{Introduction}

First-passage percolation can be thought of as a discrete model for the spread of an infectious entity. The model has been studied extensively since introduced by \cite{hamwel65}, most notably in the case when the underlying discrete structure is given by the $\Z^d$ nearest neighbour lattice for some $d\geq2$ (see e.g.\ \cite{kesten86} for the state of the art in the mid 80s, and \cite{grikes12} for the development since). Although these studies have generated a variety of tools and techniques, there are many conjectured properties of the model that so far have not been rigorously verified. Among these properties we find monotonicity of travel times, fluctuations around the asymptotic shape, and the existence of infinite geodesics (see \cite{howard04} for a more comprehensive list of open problems). The incomplete picture in the challenging higher dimensional case is the main motivation behind the present study, in which first-passage percolation on essentially 1-dimensional periodic graphs is considered. Of 
particular interest will be subgraphs of the usual $\Z^d$ lattice, such as the $\Z\times\{0,1,\ldots,K-1\}^{
d-1}$ nearest neighbour graph, for some $K,d\geq1$, below referred to as the \emph{$(K,d)$-tube}.

The asymptotic behaviour of first-passage percolation on the class of graphs considered in this study is found to be typically 1-dimensional, and differ in several aspects from what is expected in higher dimensions. Apart from a usual strong law of large numbers, we prove that the first-passage process further obeys a central limit theorem, a law of the iterated logarithm, and a Donsker theorem. In addition to these classical limit theorems, we are able to obtain a more precise description of the process, showing that both the mean and variance of travel times grow monotonically in the distance between vertices. This behaviour is plausible to hold also in higher dimensions, but remains so far unknown except in a certain special case (see \cite{gouere12}). We further show that an analogous behaviour to the one described above also holds for the $\ell^1$-length of the path with minimal weight between two vertices. Finally, the regenerative idea is used to construct a coupling between two first-passage 
processes with different initial configurations. As an application of the coupling, we deduce a 0--1 law.

The periodicity of the graphs we consider induces a form of translation invariance of the first-passage process. Together with the 1-dimensional nature of the graphs, this allows for the identification of a suitable regenerative structure. The regenerative structure may then be exploited to study the first-passage process with the help of classical renewal theory. This paper aims at giving a clear demonstration of how the identification of a regenerative behaviour can provide detailed information about a random process. The same regenerative idea has also been found useful in the study of first-passage percolation on the integer lattice in two and more dimensions (see \cite{A13,A13-2}).

The asymptotic results derived in this study greatly generalizes earlier results by \cite{schlemm11}, who proves a central limit theorem for first-passage percolation on the $(2,2)$-tube with exponential weights. A central limit theorem and Donsker theorem have also been obtained independently and by different methods by \cite{chadey09} (further commented upon below). Finally, let us mention that although the rate of growth on either $\Z^d$ for $d\geq2$, or the $(K,d)$-tube seems practically impossible to compute, this has successfully been considered for some `width-two stretches' such as the $(2,2)$-tube by \cite{schlemm09,renlund10,renlund11,flagamsor11}. Before we present the contributions of this study further, we introduce some basic definitions and provide the reader with some background.

\paragraph{Description of model}

Let $\G=(\Vb,\Eb)$ be a connected graph, and associate to the edges of the graph non-negative i.i.d.\ random weights $\{\tau_e\}_{e\in\Eb}$, referred to as \emph{passage times}. Passage times are interpreted as the time it takes an infection to traverse the edge. To avoid trivialities we assume throughout that the passage time distribution $\Pt(\;\cdot\;):=P(\tau_e\in\cdot\;)$ does not concentrate all mass at a single point. Let us by a \emph{path} refer to an alternating sequence of vertices and edges; $v_0,e_1,v_1,\ldots,e_m,v_m$,
such that the vertex $v_k$ is the endpoint of the edges $e_k$ and $e_{k+1}$.
A path with one endpoint in $U$ and the other in $V$, where $U,V\subset\Vb$, will be referred to as a path from $U$ to $V$. We will repeatedly abuse notation and identify a path with its set of edges, and occasionally with its set of vertices. For a path $\Gamma$, we define the passage time of $\Gamma$ as $T(\Gamma):=\sum_{e\in\Gamma}\tau_e$, and define the \emph{travel time} between two sets $U,V\subset\Vb$ as
$$
T(U,V):=\inf\{T(\Gamma):\Gamma\text{ is a path from }U\text{ to }V\}.
$$
In the case $U=\{u\}$ or $V=\{v\}$, we simply suppress brackets.

First-passage percolation refers to the process started with a finite set $I\subset\Vb$ of infected vertices, from which the infection spreads to adjacent vertices with delays indicated by the passage times. The corresponding set of infected vertices at time $t$, is given by
$$
B_t:=\big\{v\in\Vb:T(I,v)\leq t\big\}.
$$

\paragraph{Higher dimensional background}

The foremost characteristic feature of first-passage percolation is its subadditive behaviour, meaning that for all vertices $u$, $v$ and $w$ of the graph
$$
T(u,v)\,\leq\, T(u,w)+T(w,v).
$$
Its importance was realised already by \cite{hamwel65}, and further inspired \cite{kingman68} to derive his Subadditive Ergodic Theorem. Consider in the following first-passage percolation on the $\Z^d$ lattice (for $d\geq2$), and let $Y$ denote the minimum of $2d$ independent random variables distributed according to $\Pt$. When $\E[Y]<\infty$, it follows the Subadditive Ergodic Theorem that there is a constant $\mu(\mathbf{e}_1)$, referred to as the \emph{time constant}, such that
\be\label{int:timeconstant}
\lim_{n\raw\infty}\frac{T(I,\mathbf{n})}{n}=\mu(\mathbf{e}_1),\quad\text{almost surely and in }L^1,
\ee
where $\mathbf{e}_1=(1,0,\ldots,0)$, and $\mathbf{n}=n\mathbf{e}_1$. A necessary and sufficient condition for simultaneous convergence in all directions was provided by \cite{coxdur81}, in inspiration of a result due to \cite{richardson73}. For the sake of convenience, let us identify $B_t$ by the set in $\R^d$ obtained by centering a unit cube around each point $z\in B_t$. Under the assumption that $\E[Y^d]<\infty$, the result of Cox and Durrett, known as the Shape Theorem, states that:
If $\mu(\mathbf{e}_1)>0$, then there exists a deterministic compact and convex set $B^\ast\subset\R^d$ with nonempty interior such that for all $\epsilon>0$, almost surely,
\be\label{eq:shapeinclusion}
(1-\epsilon)B^\ast\,\subset\,\frac{1}{t}B_t\,\subset\,(1+\epsilon)B^\ast,\quad\text{for }t\text{ large enough}.
\ee
If $\mu(\mathbf{e}_1)=0$ and $K\subset\R^d$ is compact, then, almost surely, $K\subset\frac{1}{t}\tB_t$ for all $t$ large enough.
It was further shown by \citet{kesten86} that $\mu(\mathbf{e}_1)=0$ if and only if $\Pt(0)\geq p_c(d)$,
where $p_c(d)$ is the critical value for independent bond percolation on the $\Z^d$ lattice.

The nature of fluctuations around the asymptotic shape have turned out to be harder to understand. Apart from a result by \cite{keszha97} in the `critical' case when $d=2$ and $\Pt(0)=p_c(2)=1/2$, precise results remain to a large extent unknown. In an earlier study, \citet{kesten93} showed that for $d\geq2$, if $\Pt(0)<p_c$ and $\E[\tau_e^2]<\infty$, then there are constants $C_1>0$ and $C_2<\infty$ such that
$$
C_1\;\leq\;\Var\big(T(I,\mathbf{n})\big)\;\leq\; C_2n,\quad\text{for all }n\geq1.
$$
The upper bound has since been improved, first by \cite{benkalsch03} in the special case of $\{a,b\}$-valued passage times, where $0<a<b<\infty$, and later for more general distributions by \citet{benros06,damhansos13}. These studies all give an upper bound of order $n/\log n$ for $d\geq2$. However, the predicted truth suggests that for $d=2$ the correct order of growth is $n^{2/3}$, and it is not clear which behaviour to expect in higher dimensions. For $d=2$ \citet{newpiz95} have shown, given that the passage-time distribution does not have a too large point mass at $\inf\{x\geq0:\Pt([0,x])>0\}$, that $\Var\big(T(I,\mathbf{n})\big)\geq C\log n$ for some $C>0$. The same lower bound was found independently by \citet{pemper94}, in the case of exponential passage times.


\paragraph{Classical limit theorems on 1-dimensional graphs}

In this paper we consider first-passage percolation on essentially 1-dimen\-sion\-al periodic graphs defined as follows.

\begin{df}\label{df:1dimG}
The class of \emph{essentially 1-dimensional periodic graphs} consists of all connected graphs $\G$ that can be constructed in the following manner. Let $\{\G_n\}_{n\in\Z}$ be a sequence of identical copies of some finite connected graph, each with vertex set $\Vb_{\G_n}=\{v_{n,1},\ldots,v_{n,K}\}$ and edge set $\Eb_{\G_n}=\{e_{n,1},\ldots,e_{n,L}\}$. Fix a nonempty set $J\subset\{(i,j):1\leq i,j\leq K\}$, and connect $\G_n$ to $\G_{n+1}$ for each $n$ by adding an edge $e(v_{n,i},v_{n+1,j})$ between $v_{n,i}$ and $v_{n+1,j}$, for each $(i,j)\in J$.
We will write $\Eb_{\G_n}^\ast$ for $\Eb_{\G_n}\cup\{e(v_{n,i},v_{n+1,j}):(i,j)\in J\}$, and say that a vertex $v$ in the resulting graph $\G=(\Vb,\Eb)$ is at \emph{level $n$} if $v\in\Vb_{\G_n}$.
\end{df}

\begin{figure}[htbp]
\begin{center}
\includegraphics[width=0.9\textwidth]{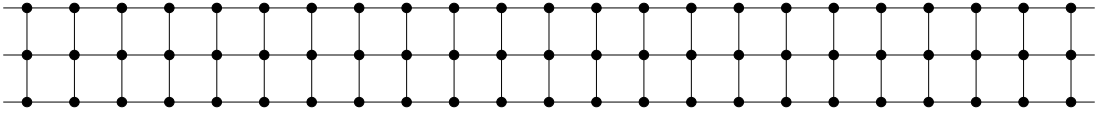}
\end{center}
\caption{
A graph of particular interest is the $\Z\times\{0,1,\ldots,K-1\}^{d-1}$ nearest neighbour graph, referred to as the $(K,d)$-tube, and here illustrated for $K=3$ and $d=2$.
}
\label{fig:3tube}
\end{figure}

Our main results concern first-passage percolation on any essentially 1-dimensional periodic graph $\G$, with passage-time distribution that does not concentrate all mass at a single point. In what follows $v_{n,i}$ will denote an arbitrary vertex at level $n$. We will prove that there are non-negative finite constants $\mu=\mu(\G,\Pt)$ and $\sigma=\sigma(\G,\Pt)$, such that the following holds.

\begin{thm}[Law of large numbers]\label{int:LLN}
If $\E[\tau_e]<\infty$, then
\be\label{eq:int:LLN}
\lim_{n\raw\infty}\frac{T(I,v_{n,i})}{n}=\mu,\quad\text{almost surely}.
\ee
If $\E[\tau_e^r]<\infty$ for some $r\geq1$, then $\big\{\big(T(I,v_{n,i})/n\big)^r\big\}_{n\geq1}$ is uniformly integrable and the convergence of (\ref{eq:int:LLN}) holds also in $L^r$.
\end{thm}

\begin{thm}[Central limit theorem]\label{int:CLT}
If $\E[\tau_{e}^{2}]<\infty$, then
$$
\frac{T(I,v_{n,i})-\mu n}{\sigma\sqrt{n}}\stackrel{d}{\raw}\chi,\quad\text{in distribution},
$$
as $n\raw\infty$, where $\chi$ has a standard normal distribution.
\end{thm}

Let $\mathcal{L}(\{x_n\}_{n\geq1})$ denote the set of limit points of a real-valued sequence $\{x_n\}_{n\geq1}$.

\begin{thm}[Law of the iterated logarithm]\label{int:LIL}
If $\E[\tau_{e}^{2}]<\infty$, then
$$
\mathcal{L}\left(\left\{\frac{T(I,v_{n,i})-\mu n}{\sigma\sqrt{2n\log\log n}}\right\}_{n\geq3}\right)=[-1,1],\quad\text{almost surely}.
$$
In particular, almost surely,
$$
\limsup_{n\raw\infty}\frac{T(I,v_{n,i})-\mu n}{\sigma\sqrt{2n\log\log n}}=1,\quad\text{and}\quad\liminf_{n\raw\infty}\frac{T(I,v_{n,i})-\mu n}{\sigma\sqrt{2n\log\log n}}=-1.
$$
\end{thm}

\begin{remark}
\emph{a)} We have at this stage preferred to state simple moment conditions, but we point out that they can be relaxed somewhat in certain cases (see Remark~\ref{rem:momentcond}).
\emph{b)} As a consequence of the regenerative structure explored in Section \ref{sec:patch}, $\mu$ and $\sigma$ are in \eqref{def:timeconstant} given by explicit formulas. It will there be clear that their values may depend on $\G$ and $\Pt$, but not on $I$ or $i$.
\emph{c)} Since $T(I,v_{n,i})$ differs from $\min_{v\in\Vb_{\G_n}}T(I,v)$ and $\max_{v\in\Vb_{\G_n}}T(I,v)$ by at most a finite number of passage times, the conclusions in Theorem~\ref{int:LLN},~\ref{int:CLT}, and~\ref{int:LIL} hold also for these quantities.
\end{remark}

The almost sure and $L^1$-convergence in Theorem~\ref{int:LLN} could easily be derived from the Subadditive Ergodic Theorem. However, it is for the understanding of our approach instructive to give an alternative proof. As the classical Central Limit Theorem for i.i.d.\ sequences extends to Donsker's theorem, Theorem~\ref{int:CLT} also extends to a functional version (see Theorem~\ref{thm:Donsker}).

At a comparison with higher dimensions, Theorem \ref{int:LLN} is the 1-dimensional analogue to the Shape Theorem. Theorems \ref{int:CLT} and \ref{int:LIL} on the other hand, point out a 1-dimensional behaviour that is not generally expected in higher dimensions. Indeed, Theorem~\ref{int:LIL} gives the precise order of fluctuations around the time constant. Restrict for a moment attention to $(K,d)$-tubes, and let $\mu_K$ and $\sigma_K$ denote the constants associated therewith. In comparison with~\eqref{eq:shapeinclusion}, Theorem~\ref{int:LIL} implies that if 
$
B^\ast(t)=[-\mu_K^{-1},\mu_K^{-1}]\times[0,K/t]^{d-1},
$
then for every $\lambda>\sigma_K\sqrt{2/\mu_K}$, almost surely,
$$
\big(1-\lambda\sqrt{t^{-1}\log\log t}\big)B^\ast(t)\,\subset\,\frac{1}{t}B_t\,\subset\,\big(1+\lambda\sqrt{t^{-1}\log\log t}\big)B^\ast(t)
$$
for large enough $t$. Moreover, both inclusions fail for $\lambda<\sigma_K\sqrt{2/\mu_K}$. The claim follows from a straightforward inversion argument, which can be found in \cite{A11thesis}.

Theorem~\ref{int:CLT} was also obtained by \cite{schlemm11} in the particular case of the $(2,2)$-tube with exponential passage times. A variant of Theorem~\ref{int:CLT} was proved by yet other means in an independent work by \cite{chadey09}. They consider first-passage percolation on the $(K,d)$-tube, where $K=K(n)$ is allowed to depend on $n$. Their main result says that for $r>2$ there exists $\alpha=\alpha(d,r)$ such that if $\E[\tau_e^r]<\infty$ and $K(n)=o(n^\alpha)$, then the sequence $\{T(\mathbf{0},\mathbf{n})\}_{n\geq1}$ continues to obey a central limit theorem. This result extends our Theorem~\ref{int:CLT}, but assumes a slightly stronger moment condition.
We emphasise that the present work was prepared simultaneously and independently of theirs, and appeared in part already as \cite{A08}, where Theorem~\ref{int:CLT} is featured.

\paragraph{Monotonicity, geodesics, and couplings}

We now return to consider arbitrary essentially 1-dimensional periodic graphs. It seems natural to believe that the expected travel time increases with the distance. A peculiar observation by \cite{vandenberg83} indicates that monotonicity of $\E[T(I,v_{n,i})]$ may not hold for all $n$. We prove that it does for large $n$.

\begin{thm}\label{int:meanT}
Let $\E[\tau_e]<\infty$. There exists $C=C(i,I,\G,\Pt)\in\R$ such that, as $n\raw\infty$,
$$
\E[T(I,v_{n,i})]\,=\,\mu n+C+o(1).
$$
\end{thm}

A direct consequence of this is that $\E[T(I,v_{n+1,i})-T(I,v_{n,i})]\raw\mu$ as $n\raw\infty$. Since $\mu>0$, this proves monotonicity of $\E[T(I,v_{n,i})]$ for large $n$. More remarkable than monotonicity of expected travel times is the fact that their fluctuations are also monotone for large $n$.

\begin{thm}\label{int:varT}
Let $\E[\tau_e^2]<\infty$. There exists $C'=C'(i,I,\G,\Pt)\in\R$ such that, as $n\raw\infty$,
$$
\Var\big(T(I,v_{n,i})\big)\,=\,\sigma^2n+C'+o(1).
$$
\end{thm}

It is interesting to relate the constants $\mu_K$ and $\sigma_K$ for the $(K,d)$-tube to similar quantities for the lattice. For fixed $d$, a trivial coupling argument shows that $\mu_K\geq\mu_{K+1}$. In fact strict inequality holds for all $K\geq1$ (see Proposition \ref{prop:muKdecreasing}), and in Proposition \ref{prop:muKlimit} we prove that
$$
\lim_{K\raw\infty}\mu_K=\mu(\mathbf{e}_1).
$$
Although it seems likely, we have not been able to verify that $\lim_{K\to\infty}\sigma_K=0$.

Apart from studying the temporal progression of a first-passage process, also describing its spatial history have received considerable attention. A simple argument showing that the infimum $T(u,v)$ is attained for some path $\gamma(u,v)$ is given in Proposition \ref{prop:geodesic}. As customary, we will use the term \emph{geodesic} to refer to the path $\gamma(u,v)$ attaining the minimal passage time. Geodesics are not necessarily unique when the passage-time distribution has atoms. For this reason, fix a deterministic rule to choose one when several are possible (e.g.\ the shortest, with some additional rule for breaking ties). Let $N(u,v)$ denote the length of the geodesic between $u$ and $v$, and similarly between finite sets $U$ and $V$. The existence of a strong law for the length of geodesics on the $\Z^d$ lattice is known only in the 'supercritical' case $\Pt(0)>p_c(d)$ (see \cite{zhazha84} and \cite{garmar04}). In the 1-dimensional case, the asymptotics presented for $\{T(I,v_{n,i})\}_{n\geq1}$ hold 
also for $\{N(I,v_{n,i})\}_{n\geq1}$, and are proven analogously. Notably, there are no moment assumptions involved. The details are left to the reader.

\begin{thm}\label{int:geodesic}
Consider first-passage percolation on any essentially 1-dimensional periodic graph. Conclusions analogous to those of Theorem \ref{int:LLN}, \ref{int:CLT}, \ref{int:LIL}, \ref{int:meanT}, \ref{int:varT}, and~\ref{thm:Donsker} hold for the sequence $\{N(I,v_{n,i})\}_{n\geq1}$ (but for different constants).
\end{thm}

In the final section of the paper we construct a coupling of two first-passage percolation infections in a way that guarantees that they will eventually coincide. As an application of the coupling we prove a 0--1 law. The construction will differ depending on $\Pt$ being continuous or discrete. We refer the reader to Section~\ref{sec:coupling} for the precise statements, and present only the 0--1 law in its continuous form here. Define the $\sigma$-algebra $\T_t:=\sigma(\{B_s\}_{s\geq t})$ and the tail $\sigma$-algebra $\T:=\cap_{t\geq0}\T_t$. We may think of $\T_t$ as the $\sigma$-algebra of events that do not depend on the times at which vertices are infected before time $t$.

\begin{thm}[0--1 law]\label{int:0--1law}
Consider first-passage percolation on an essentially 1-dimensional periodic graph. Assume that the passage time distribution has an absolutely continuous component (with respect to Lebesgue measure). Then $P(A)\in\{0,1\}$, for any event $A\in\T$.
\end{thm}

It is not known in which generality an absolutely continuous component is sufficient for the existence of a 0--1 law. But, we give an example showing that a 0--1 law analogous to Theorem~\ref{int:0--1law} cannot hold on the binary tree. An interesting case to settle would be on the $\Z^d$ lattice.

\paragraph{Outline of paper}

The main results of this paper are based on a `regenerative' nature arising for first-passage percolation on essentially 1-dimensional periodic graphs. What is meant by a regenerative behaviour, and which properties thereof that will recur throughout this paper, is clarified in the next section. Once a regenerative sequence has been identified, its asymptotics can be studied with renewal theory and stopping theory for random walks. Theorem~\ref{int:LLN}, \ref{int:CLT}, and~\ref{int:LIL} are in Section~\ref{sec:results} easily derived from the regenerative behaviour. Monotonicity of mean and variance is considered in Section~\ref{sec:monT}. Section \ref{sec:geodesic} is dedicated to properties related to the time constant. Couplings are finally constructed in Section \ref{sec:coupling}, for both continuous and discrete distributions, and the 0--1 law of Theorem~\ref{int:0--1law} is proved.

\begin{remark}
Some details are left out of this version of the paper in order to keep the presentation concise. A longer version with all details is found in~\cite{A11thesis}.
\end{remark}

\begin{remark}
It was recently pointed out to the author that some of the results in this paper could alternatively be obtained via an approach based on additive functionals of Markov chains. However, in order to clearly prove that the underlying Markov chain (in which $T(I,v_{n,i})$ would be expressed) has the required properties, it seems that the identification of a regenerative event like that in Section~\ref{sec:patch} is necessary. We have therefore preferred the regenerative approach due to its intuitive appeal.
\end{remark}

\section{Regenerative behaviour}\label{sec:patch}

The idea of how to identify a suitable regenerative sequence arises naturally for first-passage percolation with exponentially distributed passage times, which we illustrate below for the $(2,2)$-tube. A regenerative sequence will refer to the following.

\begin{df}\label{def:regseq}
We say that a sequence $\{X_k\}_{k\geq1}$ of random variables is a \emph{regenerative sequence} if there exists an increasing sequence of random variables $\{\lambda_k\}_{k\geq0}$ such that
\begin{enumerate}[\quad a)]
\item $\{\lambda_k-\lambda_{k-1}\}_{k\geq1}$ forms an i.i.d.\ sequence, and
\item $\{X_{\lambda_k}-X_{\lambda_{k-1}}\}_{k\geq1}$ forms a sequence of i.i.d.\ non-negative random variables.
\end{enumerate}
We will refer to $\{\lambda_k\}_{k\geq0}$ as a sequence of \emph{regenerative levels}.
\end{df}

\subsection{Exponential passage times}\label{sec:FPPexp}

Equip the $(2,2)$-tube with i.i.d.\ exponential passage times $\{\tau_e\}_{e\in\Eb}$, and let both vertices at level zero be initially infected. At any fixed time $t$, given the infected component $B_t$, each edge with exactly one endpoint in the infected component is equally likely to be passed by the infection next. Thus, at each level, with probability at least $1/2$, both vertices will become infected before any vertex at the following level. It follows that with probability one, at some level $r$, both vertices will become infected before any vertex at level $r+1$. Denote by $\rho$ the first level for which this happens, and let $\tau_\rho$ denote the time at which this happens. By the lack-of-memory property, the time it takes for the infection from this moment to reach $m$ levels further has the same distribution as the time it would take to reach level $m$, i.e.,
\be\label{eq:expTrho}
T(I,v_{\rho+m,i})-\tau_\rho\stackrel{d}{=}T(I,v_{m,i}).
\ee
In fact, at infinitely many levels, both vertices at that level will be infected before any vertex at higher levels. If we repeat the argument, we generate a sequence of (regenerative) levels $\{\rho_k\}_{k\geq1}$ (see Figure \ref{fig:2tubeinftree}), with corresponding sequence of instants $\{\tau_{\rho_k}\}_{k\geq1}$, such that (\ref{eq:expTrho}) holds.
\begin{figure}[htbp]
\begin{center}
\resizebox{0.9\textwidth}{!}{\input{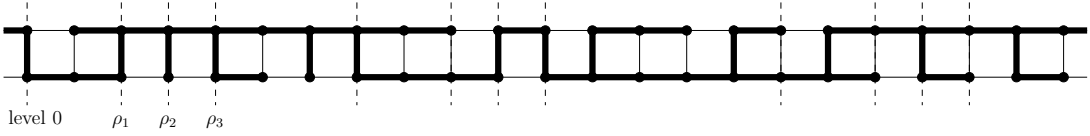}}
\end{center}
\caption{A realisation of the spread of an infection on the $(2,2)$-tube. The broken lines indicate levels at which both vertices will become infected before any vertex ahead.}
\label{fig:2tubeinftree}
\end{figure}
Since passage times are i.i.d., the consecutive differences $\rho_{k+1}-\rho_k$ will be i.i.d., as well as the differences $\tau_{\rho_{k+1}}-\tau_{\rho_k}$. It follows that $\left\{\max_{v\in\Vb_{\G_n}}T(I,v)\right\}_{n\geq1}$ is a regenerative sequence.

The point of the regenerative sequence is the following. Note that the $n$th (regenerative) level and the time at which it occurs may be written as sums of i.i.d.\ random variables, i.e.,
$$
\rho_n=\sum_{k=0}^{n-1}\rho_{k+1}-\rho_k\quad\text{and}\quad\tau_{\rho_n}=\sum_{k=0}^{n-1}\tau_{\rho_{k+1}}-\tau_{\rho_k},
$$
where $\rho_0=0$ and $\tau_{\rho_0}=0$. This link will enable the first-passage process to be studied via the classical theory for i.i.d.\ sequences.

\subsection{The general case}\label{sec:FPPG}

Next, consider first-passage percolation with general passage time distribution on any essentially 1-dimensional periodic graph. In what follows, an edge \emph{at} level $n$ refers to an edge in $\Eb_{\G_n}$. An edge \emph{between} levels $n$ and $n+m$ refers to an edge in $\Eb_{\G_n}^\ast\cup\ldots\cup\Eb_{\G_{n+m-1}}^\ast\cup\Eb_{\G_{n+m}}$. We first define our regenerative event. Let $M$ be a positive integer and denote the set of edges between level $n$ and $n+2M$ by $E_n$. Fix a path $\gamma_n$ of shortest length between $\Vb_{\G_n}$ and $\Vb_{\G_{n+2M}}$, i.e., between two vertices at level $n$ and $n+2M$, respectively. Define the subset $\hat{E}_n$ of $E_n$ as
\be\label{eq:hEn}
\hat{E}_n:=\gamma_n\cup\Eb_{\G_n}\cup\Eb_{\G_{n+2M}}.
\ee
Define
\be\label{eq:mtMt}
\begin{aligned}
\mt&:=\inf\big\{x\geq0:\Pt\big([0,x]\big)>0\big\}\\
\Mt&:=\sup\big\{x\geq0:\Pt\big([x,\infty)\big)>0\big\}.
\end{aligned}
\ee
Note that $0\leq\mt<\Mt\leq\infty$ since we consider passage-time distributions not concentrated to a single point. Fix $\tlow$ and $\thi$ such that $\mt<\tlow<\thi<\Mt$, define the regenerative event
\be\label{eq:Ank}
A_n:=\big\{\tau_e\leq\tlow, \forall e\in\hat{E}_n\big\}\cap\big\{\tau_e\geq\thi, \forall e\in E_n\setminus\hat{E}_n\big\}.
\ee
The event $A_n$ is depicted in Figure \ref{fig:GAnk}. Trivially $P(A_n)>0$. The vertex at level $n+M$ first reached via $\gamma_n$ will be of particular interest, so we introduce the following notation.

\begin{df}
Let $\hat{v}_n$ denote the vertex at level $n$ first reached via $\gamma_{n-M}$. That is, $\hat{v}_{n+M}$ denotes the vertex at level $n+M$ first reached via $\gamma_n$.
\end{df}

\begin{figure}[htbp]
\begin{center}
\resizebox{0.9\textwidth}{!}{\input{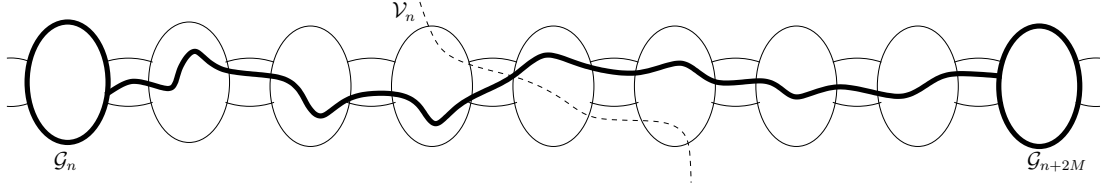}}
\end{center}
\caption{If $A_n$ occurs, the thick edges symbolising $\hat{E}_n$ are 'quick'.}
\label{fig:GAnk}
\end{figure}

We will consider random variables conditioned on the occurrence of events like $A_n$. Two random variables $X$ and $Y$ will be said to be \emph{conditionally independent} given $A$, if the random variables $X$ conditioned on $A$, and $Y$ conditioned on $A$, are independent.

\begin{lma}\label{lma:Ank}
For every $\tlow$ and $\thi$ such that $\mt<\tlow<\thi<\Mt$, there exists $M\in\N$, such that for all $u\in\bigcup_{k\leq n}\Vb_{\G_k}$ and $v\in\bigcup_{k\geq n+2M}\Vb_{\G_k}$:
\begin{enumerate}[\quad a)]
\item If $A_n$ occurs, then $T(\Gamma)>T(u,v)$ for any path $\Gamma$ between $u$ and $v$ not visiting $\hat{v}_{n+M}$, and
\be\label{eq:lma:Ank}
T(u,v)=T(u,\hat{v}_{n+M})+T\left(\hat{v}_{n+M},v\right).
\ee
\item $T(u,\hat{v}_{n+M})$ and $T\left(\hat{v}_{n+M},v\right)$ are conditionally independent given $A_n$. In addition, given $A_n$, $T(u,\hat{v}_{n+M})$ is conditionally independent of the passage time of any edge beyond level $n+2M$, and $T(\hat{v}_{n+M},v)$ is conditionally independent of the passage time of any edge before level $n$.
\end{enumerate}
\end{lma}

\begin{proof}
It suffice to prove the lemma for $u\in\Vb_{\G_n}$ and $v\in\Vb_{\G_{n+2M}}$. For given $\tlow$ and $\thi$, choose an integer $M>\tlow L/(\thi-\tlow)$,
where $L$ denotes the cardinality of $\Eb_{\G_n}$. Set $\beta:=\dist(\hat{v}_{n+M},\Vb_{\G_{n+2M}})$, where $\dist(v,V)$ denotes the smallest number of edges one has to pass in order to reach a vertex of $V$ from $v$, and define (see Figure \ref{fig:GAnk})
$$
\Vc_n:=\bigg\{v\in\bigcup_{j=n}^{n+2M}\Vb_{\G_j}:\dist(v,\Vb_{\G_{n+2M}})=\beta\bigg\}.
$$
We will prove that, given $A_n$,
\be\label{eq:Vcineq}
T(u,\hat{v}_{n+M})<T(u,w)\quad\text{and}\quad T(\hat{v}_{n+M},v)<T(w,v)
\ee
for all $w\in\Vc_{n}\setminus\{\hat{v}_{n+M}\}$. This proves that $T(\Gamma)>T(u,v)$ for all paths $\Gamma$ between $u$ and $v$ that does not visit $\hat{v}_{n+M}$, since each path from $u$ to $v$ has to pass some vertex in $\Vc_n$. Thus, (\ref{eq:lma:Ank}) holds. That $T(u,\hat{v}_{n+M})$ and $T\left(\hat{v}_{n+M},v\right)$ are conditionally independent given $A_n$ is easily seen from the following observation. When $A_n$ occurs, it follows from (\ref{eq:Vcineq}) that $T(u,\hat{v}_{n+M})$ is the infimum of $T(\Gamma)$ over all paths $\Gamma$ from $u$ to $\hat{v}_{n+M}$ that intersects $\Vc_n$ only at $\hat{v}_{n+M}$, whereas $T\left(\hat{v}_{n+M},v\right)$ is the infimum of $T(\Gamma)$ over all paths $\Gamma$ from $\hat{v}_{n+M}$ to $v$ that intersects $\Vc_n$ only at $\hat{v}_{n+M}$. Hence, the infima of passage times are taken over paths in disjoint parts of the graph. The remaining statement in \emph{b)} follows similarly.

To deduce (\ref{eq:Vcineq}), condition on $A_n$. First note that, by definition of $\gamma_n$ and $\Vc_n$, $T(w^\prime,\hat{v}_{n+M})<T(w^\prime,w)$ for any vertex $w^\prime$ visited by $\gamma_n$, and $w\in\Vc_n\setminus\{\hat{v}_{n+M}\}$. Let $\gamma_n^-$ denote the part of the path $\gamma_n$ between $\Vb_{\G_n}$ and $\hat{v}_{n+M}$. Let $\Gamma$ be any path from $u$ to $\Vc_n$ disjoint from $\gamma_n^-$. Note that
$$
T(u,\hat{v}_{n+M})\,\leq\,\left(L+|\gamma_n^-|\right)\tlow\quad\text{and}\quad T(\Gamma)\,\geq\,|\gamma_n^-|\thi.
$$
(Here $\gamma_n^-$ is identified with its set of edges.) Thus, by the choice of $M$,
\bea
T(\Gamma)-T(u,\hat{v}_{n+M})\;\geq\;(\thi-\tlow)|\gamma_n^-|-\tlow L\;\geq\;(\thi-\tlow)M-\tlow L\;>\;0.
\eea
This proves that $T(u,\hat{v}_{n+M})<T(u,w)$ for all $w\in\Vc_{n}\setminus\{\hat{v}_{n+M}\}$. The proof of the remaining inequality in (\ref{eq:Vcineq}) is similar.
\end{proof}

Assume from now on that $\tlow$, $\thi$ and $M$ are chosen in accordance with Lemma \ref{lma:Ank}. Next, introduce an auxiliary random variable $\Delta$. Throughout, $\Delta$ will denote any bounded $\N$-valued random variable independent of $\{\tau_e\}_{e\in\Eb}$. Except for in Section~\ref{sec:monT}, we will assume $\Delta\equiv0$.

Let $\rho_I:=\max\{n\in\Z:\Vb_{\G_n}\cap I\neq\emptyset\}$ denote the furthest initially infected level. Define
$$
n_k:=\rho_I+\Delta+k(2M+1),\quad\text{for }k\in\Z,
$$
and note that the sequence of events $\{\Ank\}_{k\in\Z}$ is readily seen to be i.i.d. Let $\kappa=\min\{k\geq0:A_{n_k}\text{ occurs}\}$ and set $\rho_0:=n_\kappa+M$.
Define further
$$
\rho_k\;:=\;M+\min\{n_m:n_m>\rho_{k-1}\text{ and }A_{n_m}\text{ occurs}\},\quad\text{for }k\geq1.
$$
(And analogously for $k<0$.) Since $\rho_k\geq\rho_I+M$ for all $k\geq0$, Lemma \ref{lma:Ank} says that each path along which any vertex at level $\rho_k+M$ and beyond is infected has to pass the vertex $\hat{v}_{\rho_k}$.
\begin{figure}[htbp]
\begin{center}
\resizebox{0.9\textwidth}{!}{\input{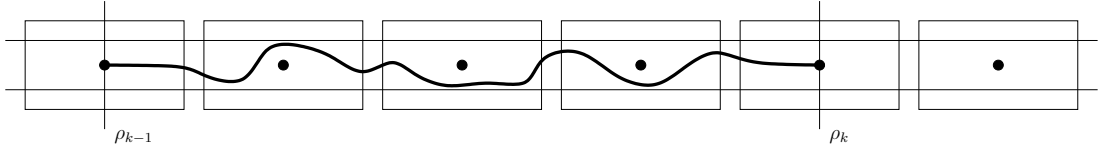}}
\end{center}
\caption{A schematic picture in which boxes indicate locations of the sequence $\{\Ank\}_{k\geq0}$, vertical lines indicate the sequence $\{\rho_k\}_{k\geq0}$, and dots indicate $\{\hat{v}_{n_k}\}_{k\geq0}$. The distance between the two vertical lines is $S_k$, and the thick curve indicates $\tSk$.}
\label{fig:regevent}
\end{figure}

\begin{df}
The vertex $\hat{v}_n$ is referred to as a \emph{regeneration point} if $n=\rho_k$ for some $k\geq0$.
\end{df}

For $k\geq1$, let $S_k$ denote the distance (measured in levels) between two regeneration points, and $\tSk$ denotes the passage time between two regeneration points. That is, 
$$
S_k:=\rho_k-\rho_{k-1},\quad\text{and}\quad\tSk:=T(\hat{v}_{\rho_{k-1}},\hat{v}_{\rho_k}).
$$
By Lemma \ref{lma:Ank} we see that $\tSk=T(I,\hat{v}_{\rho_k})-T(I,\hat{v}_{\rho_{k-1}})$, which immediately gives that
$$
\rho_n=\rho_0+\sum_{k=1}^{n}S_k,\quad\text{and}\quad T(I,\hat{v}_{\rho_n})=T(I,\hat{v}_{\rho_0})+\sum_{k=1}^{n}\tSk.
$$

\begin{lma}\label{lma:patch}
Assume that $\tlow$, $\thi$ and $M$ are chosen in accordance with Lemma \ref{lma:Ank}. Then, $\{(\tSk, S_k)\}_{k\in\Z}$ forms a sequence of i.i.d.\ $[0,\infty)\times\Z_+$-valued random variables.
\end{lma}

\begin{proof}
It is easily seen that $\{S_k\}_{k\in\Z}$ is an i.i.d.\ sequence of geometrically distributed random variables, times a factor $2M+1$, since the events $\Ank$ are pairwise independent with equal success probabilities. For the same reason does the distribution of $\tSk$ not depend on $k$. Independence of $\tSk$ and $\tau_{S_l}$ for $k<l$ follows from Lemma~\ref{lma:Ank} part~\emph{b)}, since $\tau_{S_l}$ is independent of $\rho_k$.
\end{proof}

\begin{prop}\label{prop:regenerative}
$\{T(I,\hat{v}_n)\}_{n\geq1}$ is a regenerative sequence. Moreover, if $\tlow$, $\thi$ and $M$ are chosen in accordance with Lemma \ref{lma:Ank}, then $\{\rho_n\}_{n\geq0}$ is a sequence of regenerative levels, and
$$
T(I,v_{\rho_n+m,i})-T(I,\hat{v}_{\rho_n})\,\stackrel{d}{=}\,T(I,v_{\rho_1+m,i})-T(I,\hat{v}_{\rho_1}),\quad\text{for all }m\geq M,n\geq1,
$$
where superscript $d$ indicates that the equality holds in distribution.
\end{prop}

\begin{proof}
That $\{T(I,\hat{v}_n)\}_{n\geq1}$ is a regenerative sequence with regenerative levels $\{\rho_n\}_{n\geq0}$ follows from Lemma \ref{lma:patch}. By Lemma \ref{lma:Ank}, $T(I,v_{\rho_n+m,i})-T(I,\hat{v}_{\rho_n})$ equals $T(\hat{v}_{\rho_n},v_{\rho_n+m,i})$ for $m\geq M$, whose distribution is independent of $n$, by the i.i.d.\ nature of the events $\Ank$.
\end{proof}

Let $\mut:=\E[\tSk]$ and $\muS:=\E[S_k]$ denote the expected passage time and distance between two regeneration points, respectively, and define
\be\label{def:timeconstant}
\mu:=\frac{\mut}{\muS},\quad\text{and}\quad\sigma^2:=\frac{\Var(\tSk-\mu S_k)}{\muS}.
\ee
The constants $\mu$ and $\sigma^2$ are those featured in the introduction. In order to state clear moment conditions, we need to know how moments of $\tau_e$ relate to moments of $S_k$ and $\tSk$. The next result shows that $\E[\tau_e^\alpha]<\infty$, for $\alpha=1,2$ respectively, is sufficient for $0<\mu<\infty$ and $0<\sigma^2<\infty$.

\begin{prop}\label{prop:pmoments}
Assume that $\Pt$ is not concentrated to a single point. Then,
\begin{enumerate}[\quad a)]
\item there exists $\alpha>0$ such that $\E\left[e^{\alpha S_k}\right]<\infty$.
\end{enumerate}
Assume further that there are $p\geq1$ (edge) disjoint paths from $\hat{v}_0$ to $\hat{v}_1$, and let $Y_p$ denote the minimum of $p$ independent random variables distributed according to $\Pt$. Then,
\begin{enumerate}[\quad a)]
\setcounter{enumi}{1}
\item if $\E[Y_p^{\alpha}]<\infty$, for some $\alpha>0$, we have $\E\big[\tSk^\alpha\big]<\infty$.
\end{enumerate}
\end{prop}

\begin{proof}
\emph{a)} Recall that $S_k/(2M+1)$ is geometrically distributed with parameter $p_A=P(A_n)$. So
\bea
\E\big[e^{\alpha S_k}\big]\;=\;\sum_{n=1}^{\infty}e^{\alpha(2M+1)n}(1-p_A)^{n-1}p_A\; =\;e^{\alpha(2M+1)}p_A\sum_{n=1}^{\infty}\left(e^{\alpha(2M+1)}(1-p_A)\right)^{n-1},
\eea
which is finite given that $e^{\alpha(2M+1)}(1-p_A)<1$.

\emph{b)}
Let $\eta=\{\eta_e\}_{e\in\Eb}$ denote the family of indicators where $\eta_e$ takes on the values $-1,0,1$, depending on whether $\{\tau_e\leq\tlow\}$, $\{\tau_e\in(\tlow,\thi)\}$ or $\{\tau_e\geq\thi\}$. Independently of $\{\tau_e\}_{e\in\Eb}$, let $\{\tilde{\tau}_e\}_{e\in\Eb}$ be an i.i.d.\ collection of random variables distributed as $P(\tilde{\tau}_e\in\cdot)=P(\tau_e\in\cdot|\tau_e\ge\thi)$, and define $\{\sigma_e\}_{e\in\Eb}$ as
\bea
\sigma_e:=\left\{
\begin{aligned}
 &\tau_e &\text{if }\eta_e=1,\\
&\tilde{\tau}_e &\text{otherwise}.
\end{aligned}
\right.
\eea
Note that $\{\sigma_e\}_{e\in\Eb}$ is an i.i.d.\ family independent of $\eta$, but that $\eta$ determines $\{\Ank\}_{k\in\Z}$, and hence $\{\rho_j\}_{j\geq0}$. In particular, $\{\sigma_e\}_{e\in\Eb}$ and $\{\rho_j\}_{j\geq0}$ are independent. Let $T'(u,v)$ denote the passage time between $u$ and $v$ in $\Vb$ with respect to $\{\sigma_e\}_{e\in\Eb}$. By construction $\tau_e\leq\sigma_e$ for every $e\in\Eb$, so $T(u,v)\leq T'(u,v)$.

Note that $P(\sigma_e>t)\le C_1P(\tau_e>t)$ for some finite constant $C_1$. We will next prove that for every $j\in\Z$ and $\alpha>0$, and for some $C_2<\infty$,
\be\label{eq:EYbound}
\E\big[T'(\hat{v}_{j-1},\hat{v}_j)^\alpha\big]\;\leq\; C_2\E[Y_p^\alpha].
\ee
Let $\Gamma_j^{(1)},\ldots,\Gamma_j^{(p)}$ denote the $p$ disjoint paths from $\hat{v}_{j-1}$ to $\hat{v}_j$. Let $\lambda$ denote the length of the longest of these paths. Then (\ref{eq:EYbound}) follows immediately from
\bea
\begin{split}
P\big(T'(\hat{v}_{j-1},\hat{v}_j)^\alpha>t\big)\;&\leq\,\prod_{i=1,\ldots,p}P\left(T'\big(\Gamma_j^{(i)}\big)>t^{1/\alpha}\right)\;\leq\,\prod_{i=1,\ldots,p}\lambda\,P\big(\sigma_e>t^{1/\alpha}/\lambda\big)\\
&\leq\; (\lambda C_1)^pP\big(\tau_e>t^{1/\alpha}/\lambda\big)^p\;=\;(\lambda C_1)^pP\big(Y_p^\alpha>t/\lambda^\alpha\big),
\end{split}
\eea
where the second inequality follows since $T'\big(\Gamma_j^{(i)}\big)\geq s$ implies that some edge in $\Gamma_j^{(i)}$ has $\sigma_e>s/\lambda$.

Set $\Lambda_n:=\{S_k=(2M+1)n\}$. By~\eqref{eq:EYbound}, subadditivity and above domination,
we deduce that
\be\label{eq:EtSkbound}
\begin{split}
\E\big[\tSk^{\alpha}\big]\;&\leq\;\sum_{n=1}^\infty\E\bigg[\bigg.\bigg(\sum_{j=\rho_{k-1}+1}^{\rho_k}T'(\hat{v}_{j-1},\hat{v}_j)\bigg)^\alpha\bigg|\Lambda_n\bigg]P(\Lambda_n)\\
&\leq\;\sum_{n=1}^{\infty}n^\alpha\sum_{j=1}^{n}\E\left[\left.T'(\hat{v}_{j-1},\hat{v}_j)^\alpha\right|\Lambda_n\right]P(\Lambda_n)\\
&\leq\; C_2\sum_{n=1}^{\infty}n^{\alpha+1}\E\left[Y_p^\alpha\right]P(\Lambda_n)\;\leq\; C_2\E\left[Y_p^\alpha\right]\E\left[S_k^{\alpha+1}\right],
\end{split}
\ee
where the second inequality follows since for any non-negative numbers $a_j$ we have
\be\label{ieq:sum}
\bigg(\sum_{j=1}^{n}a_j\bigg)^\alpha\leq\;(n\max_ja_j)^\alpha\;\leq\; n^\alpha\sum_{j=1}^{n}a_{j}^{\alpha}.
\ee
Thus, $\E\big[\tSk^\alpha\big]<\infty$ from part \emph{a)}.
\end{proof}

\section{Asymptotics for first-passage percolation}\label{sec:results}

The regenerative behaviour, and in particular that $\{\tSk\}_{k\geq1}$ and $\{S_k\}_{k\geq1}$ form i.i.d.\ sequences, will prove helpful when deriving asymptotics for $\{T(I,\hat{v}_n)\}_{n\geq1}$. The approach will include stopping the sequence $\{T(I,\hat{v}_{\rho_k})\}_{k\geq0}$ in a suitable way. Such object has been thoroughly studied (see e.g.\ \cite{gut09}), and we will occasionally appeal to known results in order to keep the presentation concise. However, it is often both easy and instructive to derive our results directly from the regenerative behaviour.

Assume throughout that $\tlow$, $\thi$ and $M$ are chosen in accordance with Lemma \ref{lma:Ank}, and that $\Delta\equiv0$. We will without further comment use the fact that if $X_n\raw X$ and $\eta_n\raw\infty$ almost surely as $n\raw\infty$, then $X_{\eta_n}\raw X$ almost surely as $n\raw\infty$. We also remind the reader that for any i.i.d.\ sequence $\{X_n\}_{n\geq 1}$, a simple application of the Borel-Cantelli lemmas shows that
\be\label{eq:asconvergence}
\lim_{n\raw\infty}\frac{X_n^\alpha}{n}=0,\text{ almost surely}\quad\Leftrightarrow\quad\E[|X_1|^\alpha]<\infty.
\ee
In order to approximate $T(I,\hat{v}_n)$, we will stop the regenerating sequence when $\Ank$ occurs for the least $k$ such that $n_k\geq n$. In terms of the sequence of regenerative levels, we define
$$
\nu(n):=\min\{m\geq0:\rho_m\geq n+M\}.
$$

\begin{lma}\label{lma:stopptid}
$\{\nu(n)\}_{n\geq0}$ is a non-decreasing sequence such that almost surely
$$
\lim_{n\raw\infty}\frac{n}{\nu(n)}=\muS\quad\text{and}\quad\lim_{n\raw\infty}\frac{\rho_{\nu(n)}}{n}=1.
$$
\end{lma}

\begin{proof}
It is clear that $\nu(n)\uparrow\infty$ as $n\raw\infty$, almost surely. Lemma \ref{lma:patch} and Proposition \ref{prop:pmoments} assure that $\{S_k\}_{k\geq1}$ forms an i.i.d.\ sequence with finite mean. Since $\rho_{\nu(n)-1}<n+M\leq\rho_{\nu(n)}$, we have
$$
\frac{\rho_{\nu(n)}}{\nu(n)}-\frac{S_{\nu(n)}}{\nu(n)}\;<\;\frac{n+M}{\nu(n)}\;\leq\;\frac{\rho_{\nu(n)}}{\nu(n)}.
$$
This, together with the classical Law of Large Numbers and (\ref{eq:asconvergence}) proves the first statement. Since
$$
\frac{\rho_{\nu(n)}}{n}\,=\,\frac{\rho_{\nu(n)}}{\nu(n)}\frac{\nu(n)}{n},
$$
the second statement follows from the Law of Large Numbers and the first .
\end{proof}

To prove Theorem \ref{int:LLN}, \ref{int:CLT} and \ref{int:LIL} it suffices to consider the sequence $\{T(I,\hat{v}_n)\}_{n\geq1}$, since $T(I,v_{n,i})$ and $T(I,\hat{v}_n)$ differ by at most a finite sum of random variables. In fact, it will be sufficient to study $\{\Trnn\}_{n\geq1}$, and apply the following lemma to finish. Recall that $Y_p$ denotes the minimum of $p$ independent random variables distributed according to $\Pt$.

\begin{lma}\label{lma:approx}
Assume that there are $p\geq1$ (edge) disjoint paths from $\hat{v}_0$ to $\hat{v}_1$.
For every $\alpha>0$,
\begin{enumerate}[\quad a)]
\item $\displaystyle{\lim_{n\raw\infty}\frac{|\rho_{\nu(n)}-n|^\alpha}{n}=0}$, almost surely.
\item if $\E[Y_p^{\alpha}]<\infty$, then $\displaystyle{\lim_{n\raw\infty}\frac{|T(I,\hat{v}_n)-\Trnn|^\alpha}{n}=0}$, almost surely.
\end{enumerate}
\end{lma}

\begin{proof}
Since $\rho_{\nu(n)}-n\leq S_{\nu(n)}+M\stackrel{d}{=}S_k+M$, then \emph{a)} follows from (\ref{eq:asconvergence}) and part \emph{a)} of Proposition \ref{prop:pmoments}. By subadditivity
$$
|\Trnn-T(I,\hat{v}_n)|\;\leq\sum_{j=n+1}^{\rnn}T(\hat{v}_{j-1},\hat{v}_j)\;\leq\sum_{j=\rho_{\nu(n)-1}-M+1}^{\rho_{\nu(n)}}T(\hat{v}_{j-1},\hat{v}_j),
$$
which in the proof of Proposition \ref{prop:pmoments} was seen to have finite moment of the same order as $Y_p$. Thus, also \emph{b)} follows from (\ref{eq:asconvergence}).
\end{proof}

\subsection{Proof of point-wise limit theorems}

\begin{proof}[{\bf Proof of Theorem \ref{int:LLN}}]
\emph{Almost sure convergence}.
Lemma \ref{lma:patch} and Proposition~\ref{prop:pmoments} give that $\{\tSk\}_{k\geq1}$ is an i.i.d.\ sequence with finite mean. Thus, as $n\raw\infty$,
$$
\frac{\Trnn}{n}\;=\;\frac{T(I,\hat{v}_{\rho_0})+\sum_{k=1}^{\nu(n)}\tSk}{\nu(n)}\frac{\nu(n)}{n}\;\raw\;\frac{\mut}{\muS},\quad\text{almost surely},
$$
by the classical Law of Large Numbers and Lemma \ref{lma:stopptid}. Thus by Lemma~\ref{lma:approx}, as $n\raw\infty$,
$$
\frac{T(I,\hat{v}_n)}{n}\;=\;\frac{\Trnn}{n}+\frac{T(I,\hat{v}_n)-\Trnn}{n}\;\raw\;\frac{\mut}{\muS},\quad\text{almost surely}.
$$

\emph{Uniform integrability}.
Assume that $\E[\tau_e^r]<\infty$. Since the distributions of $T(I,\hat v_{\rho_0})$ and $\big(T(I,v_{n,i})-\Trnn\big)^r$ are independent of $n$ and have finite mean, it suffices to show that $\big\{\big(\frac{1}{n}T(\hat v_{\rho_0},\hat v_{\rho_{\nu(n)}})\big)^r\big\}_{n\geq1}$ is uniformly integrable. However, by subadditivity
$$
\frac{1}{n}T(\hat v_{\rho_0},\hat v_{\rho_{\nu(n)}})\;\leq\;\frac{1}{n}T(\hat v_{\rho_0},\hat v_{\rho_n})\;\leq\;\frac{1}{n}\sum_{j=1}^nT(\hat v_{\rho_{j-1}},\hat v_{\rho_j}),
$$
so uniform integrability follows from uniform integrability for i.i.d.\ sequences.

\emph{$L^r$-convergence}.
Immediate from the almost sure convergence and uniform integrability.
\end{proof}

Theorem \ref{int:CLT} will be deduced from the following result sometimes referred to as \emph{Anscombe's theorem}. For a proof, we refer the reader to e.g.\ \citet[Theorem 1.3.1]{gut09}.

\begin{lma}[Anscombe's theorem]
Let $\{\xi_k\}_{k\geq1}$ be an i.i.d.\ sequence with mean zero and variance $\sigma_\xi^2$. Assume further that $\eta(n)/n\stackrel{p}{\raw}\theta$ in probability as $n\raw\infty$. Then, as $n\raw\infty$,
$$
\frac{1}{\sigma_\xi\sqrt{\theta n}}\sum_{k=1}^{\eta(n)}\xi_k\stackrel{d}{\raw}\chi,\quad\text{in distribution},
$$
where $\chi$ has a standard normal distribution.
\end{lma}

\begin{proof}[{\bf Proof of Theorem \ref{int:CLT}}]
It follows from Lemma \ref{lma:patch} and Proposition \ref{prop:pmoments} that $\{\tSk-\mu S_k\}_{k\geq1}$ is an i.i.d.\ sequence with zero mean and finite variance. Anscombe's theorem, together with Lemma~\ref{lma:stopptid}, gives convergence in distribution of the former term in the right-hand side of
$$
\frac{T(I,\hat{v}_n)-\mu n}{\sigma\sqrt{n}}\;=\;\frac{\Trnn-\mu\rnn}{\sigma\sqrt{n}}+\frac{T(I,\hat{v}_n)-\Trnn-\mu(n-\rnn)}{\sigma\sqrt{n}},
$$
to a standard normal distribution, as $n\raw\infty$. The latter vanishes according to Lemma~\ref{lma:approx}.
\end{proof}


\begin{proof}[{\bf Proof of Theorem \ref{int:LIL}}]
Recall that $\tSk-\mu S_k$ are i.i.d.\ for $k\geq1$, with zero mean and finite variance, due to Lemma \ref{lma:patch} and Proposition \ref{prop:pmoments}. Trivially
\be\label{eq:LILsplit}
\begin{split}
\frac{T(I,\hat{v}_n)-\mu n}{\sigma\sqrt{2n\log\log n}}\;&=\;\frac{\Trnn-\mu\rnn}{\sigma\sqrt{\muS2\nu(n)\log\log\nu(n)}}\sqrt{\muS\frac{\nu(n)}{n}}\sqrt{\frac{\log\log\nu(n)}{\log\log n}}\\
&\quad\;+\;\frac{T(I,\hat{v}_n)-\Trnn-\mu(n-\rnn)}{\sigma\sqrt{2n\log\log n}}.
\end{split}
\ee
Since $\nu(n)$ is non-decreasing, and for each $m\in\Z_+$, there is an $n\in\Z_+$ such that $\nu(n)=m$, it follows from (an extended version of) the Law of the Iterated Logarithm for i.i.d.\ sequences that
$$
\mathcal{L}\left(\left\{\frac{\Trnn-\mu\rnn}{\sigma\sqrt{\muS2\nu(n)\log\log\nu(n)}}\right\}_{n:\nu(n)\geq3}\right)\;=\;[-1,1],\quad\text{almost surely},
$$
and, in particular, that almost surely
\bea
\limsup_{n\raw\infty}\frac{\Trnn-\mu\rnn}{\sigma\sqrt{\muS2\nu\log\log\nu}}=1,\quad\text{and}\quad\liminf_{n\raw\infty}\frac{\Trnn-\mu\rnn}{\sigma\sqrt{\muS2\nu\log\log\nu}}=-1.
\eea
Lemma \ref{lma:stopptid} gives that $\muS\nu(n)/n\raw1$, almost surely, as $n\raw\infty$, and we further conclude that
\bea
\lim_{n\raw\infty}\frac{\log\log\nu(n)}{\log\log n}=\lim_{n\raw\infty}\frac{\log\big(\log n+\log\frac{\nu(n)}{n}\big)}{\log\log n}=1,\quad\text{almost surely}.
\eea
An application of Lemma \ref{lma:approx} now completes the proof.
\end{proof}

\subsection{Functional Donsker theorem}

Let $D=D[0,\infty)$ denote the set of right-continuous functions with left-hand limits on $[0,\infty)$, and let $\mathcal{D}$ denote the $\sigma$-algebra generated by the open sets in $D$ with Skorohod's $J_1$-topology. A sequence $\{P_n\}_{n\geq1}$ of probability measures on $(D,\mathcal{D})$ is said to \emph{converge weakly} to $P$ if
$$
\int_Df\,dP_n\raw\int_Df\,dP,
$$
for all bounded, continuous $f$ from $D$ to $\R$; Denote this $P_n\stackrel{J_1}{\Rightarrow}P$. In spirit with Donsker's theorem, we prove weak convergence of travel times to Wiener measure $W$.

\begin{thm}[Functional Donsker theorem]\label{thm:Donsker}
If $\E[\tau_e^2]<\infty$, then
$$
\frac{T(I,v_{\lfloor nt\rfloor,i})-\mu\lfloor nt\rfloor}{\sigma\sqrt{n}}\stackrel{J_1}{\Rightarrow}W,\quad\text{as }n\raw\infty.
$$
\end{thm}

As for the point-wise Central Limit Theorem, there is an Anscombe version of Donsker's theorem (cf.\ \citet[Theorem 5.2.1]{gut09}), from which we will deduce Theorem \ref{thm:Donsker}.

\begin{lma}\label{lma:ADonsker}
Let $\{\xi_k\}_{k\geq1}$ be an i.i.d.\ sequence of random variables with zero mean and variance $\sigma_\xi^2$. Assume further that $\{\eta(n)\}_{n\geq0}$ is a non-decreasing
sequence of positive, integer valued random variables such that $\eta(n)/n\raw\theta$ almost surely as $n\raw\infty$. Then,
$$
\frac{1}{\sigma_\xi\sqrt{\theta n}}\sum_{k=1}^{\eta(\lfloor nt\rfloor)}\xi_k\stackrel{J_1}{\Rightarrow}W,\quad\text{as }n\raw\infty.
$$
\end{lma}

\begin{proof}[Proof of Theorem \ref{thm:Donsker}]
Lemma \ref{lma:patch} and Proposition \ref{prop:pmoments} assure that $\{\tSk-\mu S_k\}_{k\geq1}$ is an i.i.d.\ sequence with zero mean and finite variance. From Lemma \ref{lma:ADonsker} it follows that
$$
\frac{T(I,\hat{v}_{\rho_{\nu(\lfloor nt\rfloor)}})-\mu\rho_{\nu(\lfloor nt\rfloor)}}{\sigma\sqrt{n}}\stackrel{J_1}{\Rightarrow}W,\quad\text{as }n\raw\infty.
$$
It suffices to prove that, as $n\raw\infty$,
\be\label{eq:Uconvergence}
\sup_{0\leq t\leq b}\left|\frac{T(I,v_{\lfloor nt\rfloor,i})-T(I,\hat{v}_{\rho_{\nu(\lfloor nt\rfloor)}})-\mu(\lfloor nt\rfloor-\rho_{\nu(\lfloor nt\rfloor)})}{\sigma\sqrt{n}}\right|\raw0,\quad\text{almost surely}.
\ee
It follows Lemma \ref{lma:approx} that, as $n\raw\infty$,
\be\label{eq:TnTrnn}
\left|\frac{T(I,v_{n,i})-\Trnn-\mu(n-\rnn)}{\sigma\sqrt{n}}\right|\raw0,\quad\text{almost surely}.
\ee
For any sequence of real numbers $\{x_n\}_{n\geq1}$ and $b\in\R_+$ it holds that
\bea
\lim_{n\raw\infty}\frac{x_n}{\sqrt{n}}=0\quad\Leftrightarrow\quad\lim_{n\raw\infty}\frac{\max_{k\leq bn}|x_k|}{\sqrt{n}}=0.
\eea
In fact, this improves (\ref{eq:TnTrnn}), and we get
$$
\frac{\max_{k\leq bn}\left|T(I,v_{k,i})-T(I,\hat{v}_{\rho_{\nu(k)}})-\mu(k-\rho_{\nu(k)})\right|}{\sigma\sqrt{n}}\raw0,\quad\text{almost surely},
$$
as $n\raw\infty$. But this is equivalent to (\ref{eq:Uconvergence}).
\end{proof}

\begin{remark}\label{rem:momentcond}
A closer look at the proofs of Theorem \ref{int:LLN}, \ref{int:CLT}, \ref{int:LIL} and~\ref{thm:Donsker} reveals that finite moment of $\tSk$, for $k\geq1$, of correct order is sufficient for the various modes of convergence for the sequence $\{T(\hat{v}_0,\hat{v}_n)\}_{n\geq1}$. According to Proposition \ref{prop:pmoments}, $\E[\tau_e^\alpha]<\infty$ may be replaced by $\E[Y_p^\alpha]<\infty$ (for given $\alpha$), where $p$ is the number of (edge) disjoint paths from $\hat{v}_0$ to $\hat{v}_1$, and $Y_p$ denotes the minimum of $p$ independent random variables distributed according to $\Pt$.
\end{remark}

\section{Monotonicity of mean and variance}\label{sec:monT}

In this section we will prove Theorem \ref{int:meanT} and \ref{int:varT}. Throughout this section, $\Delta$ will be assumed uniformly distributed on $\{0,1,\ldots,2M\}$. Clearly, the value of $\rho_I$ has no other effect than to the value of the constants $C$ and $C'$. For ease of presentation, we will therefore assume that $\rho_I=0$. The theory for stopped random walks tells us that $\E[\Trnn]=\mu n+C$, for some constant $C$, and $\Var(\Trnn)=\sigma^2n+o(n)$, for large $n$  (see e.g.\ \citet[Theorem 4.2.4]{gut09}). An essential additional effort is needed in order to improve the latter statement to the form $\Var(\Trnn)=\sigma^2n+C'$. What then remains in order to prove Theorem \ref{int:meanT} and \ref{int:varT}, is to show that the differences between $\E[T(I,v_{n,i})]$ and $\E[\Trnn]$, and between $\Var(T(I,v_{n,i}))$ and $\Var(\Trnn)$, converge as $n\raw\infty$. We will base the proofs of Theorem \ref{int:meanT} and \ref{int:varT} on Wald's lemma.

\begin{lma}[Wald's lemma]\label{lma:Wald}
Let $\xi_1,\xi_2,\ldots$ be i.i.d.\ random variables with mean $\mu_\xi$, and set $S_n=\sum_{k=1}^n\xi_k$. Let $N$ be a stopping time with $\E[N]<\infty$.
\begin{enumerate}[\quad a)]
\item $\E[S_N]=\mu_\xi\E[N]$.
\item If $\sigma_\xi^2=\Var(\xi_1)<\infty$, then $\E\left[(S_N-\mu_\xi N)^2\right]=\sigma_\xi^2\E[N]$.
\item If $X$ is independent of $\xi_1,\xi_2,\ldots$, then $\E[XS_N]=\mu_\xi\E[XN]$. In particular, $\Cov(X,S_N)=\mu_\xi\Cov(X,N)$.
\end{enumerate}
\end{lma}

The third part of the lemma is a slight extension of the first part, and proved in an analogous way. If $\mathcal{F}_n=\sigma(\{(\rho_0,T(I,\hat{v}_{\rho_0})),(S_1,\tau_{S_1}),\ldots,(S_n,\tau_{S_n})\})$, then it is immediate from the definition that $\nu(n)$ is a stopping time with respect to the sequence of $\sigma$-algebras $\{\mathcal{F}_n\}_{n\geq1}$.

The importance of $\Delta$ is the following. Regeneration may only occur every $2M+1$ levels; Introducing a shift uniformly distributed on $\{0,1,\ldots,2M\}$ makes 'regeneration possible' equally likely for every level. The proof of the following lemma clearly show the advantage of this.

\begin{lma}\label{lma:meannu}
Assume that $\rho_I=0$. Then for $n\geq0$,
$\quad\displaystyle
\E[\nu(n)]=\frac{n}{\muS}.
$
\end{lma}

\begin{proof}
Assume that $n\geq0$. We may interpret $\nu(n)$ as the number of regeneration points before (but not including) level $n+M$, that is, the number of $k\geq0$ such that $\Ank$ occurs for $n_k<n$. Since $n_0=\Delta$, this number is at most $n_A=\big\lfloor\frac{n+2M-\Delta}{2M+1}\big\rfloor$. The shift $\Delta$ is independent of $\{\tau_e\}_{e\in\Eb}$, and conditioned on $\Delta$, we can think of $\nu(n)$ as the number of successes in 
$n_A$ independent Bernoulli trials, each with success probability $p_A=P(\Ank)$. Conditioning on $\Delta$, we see that
\be\label{eq:meannu}
\E[\nu(n)]\;=\;p_A\E\left[\left\lfloor\frac{n+2M-\Delta}{2M+1}\right\rfloor\right].
\ee
If $n-\rho_I=(2M+1)k$, for some $k\geq0$, one realise from (\ref{eq:meannu}) that
$$
\E[\nu(n)]\;=\;\frac{p_A}{2M+1}(2M+1)k\;=\;\frac{n}{\muS},
$$
where the latter equality follows from the fact that $S_k$ is geometrically distributed with parameter $p_A$, times a factor $2M+1$, that is, $\muS=(2M+1)/p_A$. Again from (\ref{eq:meannu}), one realises that as $n$ increases from $(2M+1)k$ to $(2M+1)k+2M$, then $\E[\nu(n)]$ will have to increase with $p_A/(2M+1)$ for each step.
\end{proof}

\begin{proof}[{\bf Proof of Theorem \ref{int:meanT}}]
Assume that $\rho_I=0$. Wald's lemma together with Lemma \ref{lma:meannu} gives
\bea
\E\big[\Trnn\big]\,=\,\E\bigg[T(I,\hat{v}_{\rho_0})+\sum_{k=1}^{\nu(n)}\tSk\bigg]
\,=\,\E\big[T(I,\hat{v}_{\rho_0})\big]+\mu n.
\eea
It remains to prove that there is a finite constant $C=C(i,I,\G,\Pt)$ such that
\be\label{eq:meanTnTr}
\E\big[T(I,v_{n,i})-\Trnn\big]\raw C,\quad\text{as }n\raw\infty.
\ee
Arguments of the type we will use to prove (\ref{eq:meanTnTr}) will be used repeatedly in the proof of Theorem~\ref{int:varT}. For this reason, we present the argument in detail here. To make the argument clear, we will define a random variable to which $T(I,v_{n,i})-\Trnn$ converges in distribution. The limit $C$ will then equal the expectation of this random variable.

Recall that $n_k=\Delta+k(2M+1)$ for $k\ge0$, set $m_{n,k}:=n-(2M+1)k$ for $k\geq1$, and let
\bea
\begin{aligned}
r_+&:=M+\min\{n_k\geq0:\Ank\text{ occurs}\},\\
r_0&:=M+\max\{m_{0,k}<0:A_{m_{0,k}}\text{ occurs}\}.
\end{aligned}
\eea
Observe that $r_+$ denotes the first element of the sequence $\{\rho_k\}_{k\geq0}$ greater than zero, whereas $r_0$ is not defined along the same subsequence of the integers as $\{\rho_k\}_{k\geq0}$. Introduce
\bea
Y_{k,i}:=T(\hat{v}_{r_0},v_{k,i}),\quad\text{and}\quad Y_+:=T(\hat{v}_{r_0},\hat{v}_{r_+}),
\eea
and the events
\bea
\begin{aligned}
D_{T,n}&:=\{A_{m_{n,k}}\text{ occurs for some $k$ such that }0\leq m_{n,k}<n\},\\
D_{Y,n}&:=\{A_{m_{0,k}}\text{ occurs for some $k$ such that }0\leq m_{0,k}+n<n\}.
\end{aligned}
\eea
Clearly $P(D_{T,n})=P(D_{Y,n})\raw1$ as $n\raw\infty$. Moreover,
$$
\big(T(I,v_{n,i})-\Trnn\big)\cdot\ind_{D_{T,n}}\;\stackrel{d}{=}\;\big(Y_{0,i}-Y_+\big)\cdot\ind_{D_{Y,n}}.
$$
So, if we let $T^\ast=T(I,v_{n,i})-\Trnn$ and $Y^\ast=Y_{0,i}-Y_+$, then as $n\raw\infty$,
\be\label{eq:TYconvergence}
T^\ast\;=\;T^\ast(\ind_{D_{T,n}}+\ind_{D_{T,n}^c})\;\stackrel{d}{=}\;Y^\ast+T^\ast\cdot\ind_{D_{T,n}^c}-Y^\ast\cdot\ind_{D_{Y,n}^c}\;\stackrel{d}{\to}\;Y^\ast.
\ee
If, in addition, $\big\{T(I,v_{n,i})-\Trnn\big\}_{n\geq1}$ is uniformly integrable, then the convergence carries over in mean, and we have proved~\eqref{eq:meanTnTr} with $C=\E\left[Y_{0,i}-Y_+\right]$.
To deduce uniform integrability, note that subadditivity gives
\be\label{eq:TvTvnbound}
T(I,v_{n,i})-\Trnn\;\leq\; T(v_{n,i},\hat{v}_n)+\sum_{j=n+1}^{\rnn}T(\hat{v}_{j-1},\hat{v}_j).
\ee
But, the distribution of the right-hand side of (\ref{eq:TvTvnbound}) does not depend on $n$. Thus, it suffices to see that it has finite expectation. This is easily achieved in an analogous way as in~\eqref{eq:EtSkbound} in the proof of Proposition \ref{prop:pmoments}, conditioning on $\Lambda_k=\{\rnn-n=k\}$. We omit the details.
\end{proof}

The proof of Theorem \ref{int:varT} needs considerable extra work, due to arising covariance terms. Moment convergence arguments similar to the one used to prove (\ref{eq:meanTnTr}) will be used repeatedly.

\begin{proof}[{\bf Proof of Theorem \ref{int:varT}}]
Assume that $\rho_I=0$. To begin with,
\be\label{eq:varTnTrnn}
\begin{split}
\Var\big(T(I,v_{n+2M,i})\big)\;&=\;\Var\left(\Trnn-\mu\rnn\right)\\
&\quad\;+\;\Var\left(T(I,v_{n+2M,i})-\Trnn+\mu\rnn\right)\\
&\quad\;+\;2\Cov\left(\Trnn-\mu\rnn,T(I,v_{n+2M,i})-\Trnn\right)\\
&\quad\;+\;2\mu\Cov\big(\Trnn,\rnn\big)-2\mu^2\Var\left(\rnn\right).
\end{split}
\ee
We will have to treat each of the terms on the right-hand side one by one. Consider the first term in the right-hand side in \eqref{eq:varTnTrnn}, and note that
\bea
\begin{split}
\Var\left(\Trnn-\mu\rnn\right)\;&=\;\Var\bigg(T(I,\hat{v}_{\rho_0})-\mu\rho_0+\sum_{k=1}^{\nu(n)}(\tSk-\mu S_k)\bigg)\\
&=\;\E\bigg[\bigg(\sum_{k=1}^{\nu(n)}(\tSk-\mu S_k)\bigg)^2\bigg]+\Var\big(T(I,\hat{v}_{\rho_0})-\mu\rho_0\big)\\
&\quad\;+\;2\Cov\bigg(\sum_{k=1}^{\nu(n)}(\tSk-\mu S_k),T(I,\hat{v}_{\rho_0})-\mu\rho_0\bigg)\\
\end{split}
\eea
So, an application of both part \emph{b)} and \emph{c)} of Wald's lemma, together with Lemma \ref{lma:meannu}, yield
\bea
\begin{split}
\Var\left(\Trnn-\mu\rnn\right)\;&=\;\Var\left(\tSk-\mu S_k\right)\E[\nu(n)]+\Var\big(T(I,\hat{v}_{\rho_0})-\mu\rho_0\big)\\
&\quad\;+\;\E[\tSk-\mu S_k]\Cov\big(\nu(n),T(I,\hat{v}_{\rho_0})-\mu\rho_0\big)\\
&=\;\sigma^2n+\Var\big(T(I,\hat{v}_{\rho_0})-\mu\rho_0\big).
\end{split}
\eea
Next, to conclude that $\Var(\rnn)$ is constant, interpret $\rnn$ as the level of the first regeneration after level $n$. Since a regeneration is equally likely to occur at any level, due to the shift variable $\Delta$, it follows that $\Var(\rnn)=\Var(\rnn-n)$ is independent of $n$, and therefore constant.

The remaining three terms in the right-hand side of (\ref{eq:varTnTrnn}) will in some way or other need an argument similar to that used to prove (\ref{eq:meanTnTr}). Recall the notation used for that purpose.

\begin{claim}
$\quad\displaystyle\Var\left(T(I,v_{n+2M,i})-\Trnn+\mu\rnn\right)\,=\,\Var(Y_{2M,i}-Y_++\mu r_+)+o(1)$.
\end{claim}

To prove the claim, note that we in an analogous way as in (\ref{eq:TYconvergence}) may divide into cases whether $D_{T,n}$ and $D_{Y,n}$ occurs or not, to show that as $n\to\infty$
$$
T(I,v_{n+2M,i})-\Trnn+\mu(\rnn-n)\,\stackrel{d}{\raw}\,Y_{2M,i}-Y_++\mu r_+.
$$
It remains to prove uniform integrability of $\big\{\big(T(I,v_{n+2M,i})-\Trnn+\mu(\rnn-n)\big)^2\big\}_{n\geq1}$. This can be proved similar to the uniform integrability needed for (\ref{eq:meanTnTr}), via~\eqref{eq:TvTvnbound}. It follows that for $r=1,2$, as $n\to\infty$
$$
\E\big[\big(T(I,v_{n+2M,i})-\Trnn+\mu(\rnn-n)\big)^r\big]\raw\E\big[(Y_{2M,i}-Y_++\mu r_+)^r\big],
$$
from which the claim follows.

\begin{claim}
$\quad\displaystyle\Cov\big(\Trnn,\rnn\big)\,=\,\Cov(Y_+,r_+)+o(1)$.
\end{claim}

To prove the second claim, set $r_n:=M+\max\{m_{n,k}<n:A_{m_{n,k}}\text{ occurs}\}$, and rewrite
\bea
\begin{split}
\Cov\big(\Trnn,\rnn\big)\;&=\;\Cov\big(\Trnn-T(I,\hat{v}_{r_n}),\rnn-n\big)\\
&\quad\;+\;\Cov\big(T(I,\hat{v}_{r_n}),\rnn-n\big).
\end{split}
\eea
It is easy to see that $\rnn-n\stackrel{d}{=}r_+$ for $n\geq0$. Partitioning on whether $D_{T,n}$ and $D_{Y,n}$ occur or not, we find that, as $n\raw\infty$,
$$
\Trnn-T(I,\hat{v}_{r_n})\stackrel{d}{\raw}Y_+,\quad\text{and}\quad\big(\Trnn-T(I,\hat{v}_{r_n})\big)(\rnn-n)\stackrel{d}{\raw}Y_+r_+.
$$
Uniform integrability of $\{(\rnn-n)^2\}_{n\geq1}$ and $\big\{\big(\Trnn-T(I,\hat{v}_{r_n})\big)^2\big\}_{n\geq1}$ is possible to deduce in a similar way as before, conditioning on $\Lambda_k=\{\rnn-r_n=k\}$. This implies that also $\big\{\big(\Trnn-T(I,\hat{v}_{r_n})\big)(\rnn-n)\big\}_{n\geq1}$ is uniformly integrable. We conclude that, as $n\to\infty$
\bea
\Cov\big(\Trnn-T(I,\hat{v}_{r_n}),\rnn-n\big)\;\raw\;\E\left[Y_+r_+\right]-\E\left[Y_+\right]\E\left[r_+\right]\;=\;\Cov(Y_+,r_+).
\eea
On the event $D_{T,n}$, $T(I,\hat{v}_{r_n})$ depends on passage times below level $n$, but not on $\Delta$, whereas $\rnn-n$ is independent of passage times below level $n$, and hence on $D_{T,n}$. It follows that
$$
\E\big[T(I,\hat{v}_{r_n})(\rnn-n)\cdot\ind_{D_{T,n}}\big]\,=\,\E\big[T(I,\hat{v}_{r_n})\cdot\ind_{D_{T,n}}\big]\E[r_+].
$$
In particular,
\bea
\begin{split}
\Cov\big(T(I,\hat{v}_{r_n}),\rnn-n\big)\;&=\;\E\big[T(I,\hat{v}_{r_n})(\rnn-n)\cdot\ind_{D_{T,n}}\big]-\E\big[T(I,\hat{v}_{r_n})\cdot\ind_{D_{T,n}}\big]\E[r_+]\\
&\quad\;+\;\E\big[T(I,\hat{v}_{r_n})(\rnn-n)\cdot\ind_{D_{T,n}^c}\big]-\E\big[T(I,\hat{v}_{r_n})\cdot\ind_{D_{T,n}^c}\big]\E[r_+]\\
&=\;\E\big[T(I,\hat{v}_{r_n})\big(\rnn-n-\E[r_+]\big)\cdot\ind_{D_{T,n}^c}\big].
\end{split}
\eea
As $n\raw\infty$, this expression vanishes, since on $D_{T,n}^c$ then $r_n\leq M$, and we may obtain an upper bound on $T(I,\hat{v}_{r_n})\big(\rnn-n-\E[r_+]\big)$ as in the proof of Proposition~\ref{prop:pmoments}. The claim follows.

\begin{claim}
$\quad\displaystyle\lim_{n\to\infty}\displaystyle\Cov\Big(\Trnn-\mu\rnn,T(I,v_{n+2M,i})-\Trnn\Big)$ exists finitely.
\end{claim}

Once this final claim is proved, we have proved existence of a finite constant $C'$ such that
$$
\Var\big(T(I,v_{n+2M,i})\big)\,=\,\sigma^2n+C'+o(1),\quad\text{as }n\raw\infty,
$$
and hence the theorem. To see that the claim is true, we will `reverse' our point of view in the following sense.
The sequence $\{\tSk-\mu S_k\}_{k=1}^{\nu(n)}$ has until now been considered as a sequence started at $k=1$ and stopped at $k=\nu(n)$. But, we can equally well see $\{\tSk-\mu S_k\}_{k=0}^{\nu(n)-1}$ as a sequence in the opposite direction (where $S_0=\rho_0-\rho_{-1}$ and $\tau_{S_0}=T(\hat v_{\rho_{-1}},\hat v_{\rho_0})$). That is, as a sequence started at the first point of regeneration `left' of level $n+M$, and that is stopped at the first point of regeneration `left' of level $0$.

Let $T^\ast=T(I,v_{n+2M,i})-\Trnn$. On the event $\{\nu(n)\geq1\}$, $T^\ast$ may be expressed as $T(\hat{v}_{\rho_{\nu(n)-1}},v_{n+2M,i})-T(\hat{v}_{\rho_{\nu(n)-1}},\hat{v}_{\rho_{\nu(n)}})$ and is independent of $\tSk-\mu S_k$ for $k<\nu(n)$. The event $\{\nu(n)\geq1\}$ is itself independent of $\{\tSk-\mu S_k\}_{k\geq1}$. This allows us to apply the first and later the third part of Wald's lemma to obtain
\bea
\begin{aligned}
\Cov\bigg(\sum_{k=1}^{\nu(n)}\tSk-\mu S_k,T^\ast\bigg)\;&=\;\E\bigg[\sum_{k=1}^{\nu(n)}(\tSk-\mu S_k)\,T^\ast\bigg]\;=\;\E\bigg[\sum_{k=1}^{\nu(n)}(\tSk-\mu S_k)\,T^\ast\,\ind_{\{\nu(n)\ge1\}}\bigg]\\
&=\;\E\big[(\tau_{S_{\nu(n)}}-\mu S_{\nu(n)})T^\ast\ind_{\{\nu(n)\ge1\}}\big]-\E\big[(\tau_{S_0}-\mu S_0)T^\ast\ind_{\{\nu(n)\ge1\}}\big].
\end{aligned}
\eea
Let $r_-:=M+\max\{n_k<0:\Ank\text{ occurs}\}$, $Y_-:=T(\hat{v}_{r_-},\hat{v}_{r_+})$ and $Z_{k,i}:=T(\hat{v}_{r_-},v_{k,i})$. Note that
$$
(\tau_{S_{\nu(n)}}-\mu S_{\nu(n)})\,T^\ast\cdot\ind_{\{\nu(n)\geq1\}}\;\stackrel{d}{=}\;\big(Y_--\mu(r_+-r_-)\big)\big(Z_{2M,i}-Y_-\big)\cdot\ind_H,
$$
where $H=\{\Ank\text{ occurs for some }\Delta-n\leq n_k<0\}$. That the random variables $Y_-$ and $Z_{2M,i}$ have finite variance is concluded as in the proof of Proposition~\ref{prop:pmoments} (conditioning on $\Lambda_n=\{r_+-r_-=n\}$). Thus, an application of the Monotone Convergence Theorem shows that
$$
\E\big[(\tau_{S_{\nu(n)}}-\mu S_{\nu(n)})\,T^\ast\cdot\ind_{\{\nu(n)\ge1\}}\big]\,=\,\E\big[(Y_--\mu(r_+-r_-))(Z_{2M,i}-Y_-)\big]+o(1),
$$
which is then finite. Moreover, due to the conditional independence between $\tau_{S_0}-\mu S_0$ and $T^\ast$, given $\{\nu(n)\ge1\}$, it follows that
$$
\E\big[(\tau_{S_0}-\mu S_0)\,T^\ast\cdot\ind_{\{\nu(n)\ge1\}}\big]\,=\,\E\big[(\tau_{S_0}-\mu S_0)\cdot\ind_{\{\nu(n)\ge1\}}\big]\E\big[\,T^\ast\cdot\ind_{\{\nu(n)\ge1\}}\big]P(\nu(n)\ge1)^{-1},
$$
which as $n\to\infty$ vanishes, again via an application of the Monotone Convergence Theorem.

It remains to measure the correlation between $T(I,\hat v_{\rho_0})-\mu\rho_0$ and $T^\ast$. Let $Z_{2M,i}^\prime$ and $Y_-^\prime$ be defined in the same way as $Z_{2M,i}$ and $Y_-$ above, but now for a set of passage times $\{\tau_e^\prime\}_{e\in\Eb}$ independent of $\{\tau_e\}_{e\in\Eb}$ (that defines $T(I,\hat{v}_{\rho_0})-\mu\rho_0$), but with the same $\Delta$. By conditioning on the events $\{\nu(n)\geq1\}$ (with respect to $\{\tau_e\}_{e\in\Eb}$) and $H$ (with respect to $\{\tau_e^\prime\}_{e\in\Eb}$), we see that
\bea
T^\ast\stackrel{d}{\raw}(Z_{2M,i}^\prime-Y_-^\prime)\quad\text{and}\quad
T^\ast\big(T(I,\hat{v}_{\rho_0})-\mu\rho_0\big)\stackrel{d}{\raw}(Z_{2M,i}^\prime-Y_-^\prime)\big(T(I,\hat{v}_{\rho_0})-\mu\rho_0\big),
\eea
as $n\to\infty$. That $\big\{\big(T(I,v_{n+2M,i})-\Trnn\big)^2\big\}_{n\geq1}$ is uniformly integrable was argued for during the proof of the first claim, and $T(I,\hat{v}_{\rho_0})-\mu\rho_0$ has finite variance. Consequently also $\big\{\big(T(I,v_{n+2M,i})-\Trnn\big)(T(I,\hat{v}_{\rho_0})-\mu\rho_0)\big\}_{n\geq1}$ is uniform integrable, and we have that
\bea
\Cov\big(T(I,\hat{v}_{\rho_0})-\mu\rho_0,T^\ast\big)\,=\,\Cov\big(T(I,\hat{v}_{\rho_0})-\mu\rho_0,Z_{2M,i}^\prime-Y_-^\prime\big)+o(1).
\eea
This ends the proof of the claim, and hence the theorem.
\end{proof}

\section{Properties of time constants}\label{sec:geodesic}

\subsection{Comparing a tube to the lattice}

We begin with a comparison between the $(K,d)$-tube and the $\Z^d$ lattice, for some $d\ge2$. Let $\mu_K$ denote the time constant associated with the $(K,d)$-tube, and assume that $\tau_e$ has finite mean, or that $K\ge3$ and $Y_{2d}$ has finite mean, so that $\mu_K$ is well defined. Let $\{\tau_e\}_{e\in\Eb_{\Z^d}}$ be i.i.d.\ passage times associated to the $\Z^d$ lattice, and denote by $T_K(u,v)$ the passage time with respect to $\{\tau_e\}_{e\in\Eb_{\Z^d}}$, between $u$ and $v$ when only paths visiting vertices in $\Z\times\{0,\ldots,K-1\}^{d-1}$ are allowed. This produces a simultaneous coupling of the passage time on $(K,d)$-tubes for all $K\geq1$. Trivially, $T_{K+1}(u,v)\leq T_K(u,v)$ for any $u$ and $v$ in $\Z\times\{0,\ldots,K-1\}^{d-1}$.

\begin{prop}\label{prop:muKdecreasing}
For all $K\geq1$, $\quad\mu_{K+1}<\mu_K$.
\end{prop}

\begin{proof}
Let $A_n^K$ be the event defined in (\ref{eq:Ank}) with respect to the $(K,d)$-tube, for $\gamma_n$ chosen to be the straight line segment between the points $(n,K,0,\ldots,0)$ and $(n+2M,K,0,\ldots,0)$. It follows from Lemma \ref{lma:Ank} that if $A_n^{K+1}$ occurs, then
$$
\delta:=T_K\big(n\ebf_1,(n+2M)\ebf_1\big)-T_{K+1}\big(n\ebf_1,(n+2M)\ebf_1\big)\,>\,0.
$$
Thus, if $m_k=(2M+1)k$, then
$$
T_{K+1}(\mathbf{0},m_k\ebf_1)+\delta\sum_{j=0}^{k-1}\ind_{A_{m_j}^{K+1}}\;\leq\; T_K(\mathbf{0},m_k\ebf_1)
$$
for all $k\geq0$. The claimed statement now follows dividing by $m_k$, and sending $k$ to infinity.
\end{proof}

We modify the above coupling slightly, and let $\tilde{T}_K(u,v)$ denote the passage time with respect to $\{\tau_e\}_{e\in\Eb_{\Z^d}}$, between $u$ and $v$, over paths restricted to the $\Z\times\{-K,\ldots,K\}^{d-1}$ nearest neighbour graph. This produces a simultaneous coupling of the passage time on $(2K+1,d)$-tubes.

\begin{prop}\label{prop:muKlimit}
$\quad\displaystyle{\lim_{K\raw\infty}\mu_K=\mu(\ebf_1)}$.
\end{prop}

\begin{proof}
Clearly $\tilde{T}_K(\mathbf{0},\mathbf{n})\geq\tilde{T}_{K+1}(\mathbf{0},\mathbf{n})$, and $T(\mathbf{0},\mathbf{n})=\lim_{K\raw\infty}\tilde{T}_K(\mathbf{0},\mathbf{n})$.
An application of the Monotone Convergence Theorem shows that
$$
\E\big[T(\mathbf{0},\mathbf{n})\big]\;=\;\lim_{K\raw\infty}\E\big[\tilde{T}_K(\mathbf{0},\mathbf{n})\big]\;=\;\inf_{K\geq0}\E\big[\tilde{T}_K(\mathbf{0},\mathbf{n})\big].
$$
Since the limit $\lim_{n\raw\infty}\frac{1}{n}a_n$ exists and equals $\inf_{n\geq1}\frac{1}{n}a_n$ for any subadditive real-valued sequence $\{a_n\}_{n\geq1}$, we have for any $K\geq1$ (including $K=\infty$, corresponds to the $\Z^d$ lattice) that
$$
\mu_{2K+1}\;=\;\lim_{n\raw\infty}\frac{1}{n}\E\big[\tilde{T}_K(\mathbf{0},\mathbf{n})\big]\;=\;\inf_{n\geq1}\frac{1}{n}\E\big[\tilde{T}_K(\mathbf{0},\mathbf{n})\big].
$$
Thus, since $\mu_K$ is non-increasing in $K$
\bea
\lim_{K\raw\infty}\mu_{2K+1}\;=\;\inf_{K\geq0}\inf_{n\geq1}\frac{1}{n}\E\big[\tilde{T}_K(\mathbf{0},\mathbf{n})\big]\;=\;\inf_{n\geq1}\inf_{K\geq0}\frac{1}{n}\E\big[\tilde{T}_K(\mathbf{0},\mathbf{n})\big]
\;=\;\mu(\ebf_1).\qedhere
\eea
\end{proof}

\subsection{Relating time constants to geodesics}

We return to consider arbitrary essentially 1-dimensional periodic graphs. We first deduce existence of minimising paths from Lemma \ref{lma:Ank}.

\begin{prop}\label{prop:geodesic}
Let $U$ and $V$ be two finite sets of vertices of an essentially 1-dimensional periodic graph. There is an almost surely finite path $\gamma$ from $U$ to $V$, such that
$$
T(\gamma)=T(U,V).
$$
Moreover, if the distribution $\Pt$ does not have any point masses, then $\gamma$ is almost surely unique.
\end{prop}

\begin{proof}
We may assume that $U\cup V\subseteq\bigcup_{k=0}^m\Vb_{\G_k}$. Assume further that $\tlow$, $\thi$ and $M$ are chosen in accordance with Lemma \ref{lma:Ank}. With probability one the event $A_{m+n}\cap A_{-2M-n}$ will occur for infinitely many $n\geq0$. Let $l$ be the least such $n$. It follows from Lemma \ref{lma:Ank} that for any path $\Gamma$ between $u$ and $v$ that reach beyond level $m+l+2M$ in the positive direction, or level $-2M-l$ in the negative direction, there is another path $\Gamma^\prime$ that only visits vertices in $\bigcup_{k=-2M-l}^{m+l+2M}\Vb_{\G_k}$, and that satisfies $T(\Gamma)\geq T(\Gamma^\prime)$. Thus, since there are only finitely many edges between level $-2M-l$ and $m+l+2M$, it follows that $T(U,V)$ is the minimum of the passage times over an almost surely finite number of paths. This proves the first statement. The second statement also follows from this, together with the fact that the probability of two paths having the same passage time is zero, when the passage-time 
distribution is free of point masses.
\end{proof}

As in the introduction, we will use the term \emph{geodesic} to refer to a path attaining the minimal passage time between two vertices, or two finite sets of vertices. Since geodesics are not necessarily unique, we assume a fixed deterministic rule to choose one when several are possible; Denote by $\gamma(u,v)$ this geodesic between $u$ and $v$, and let $N(u,v)$ denote its length. In particular, Theorem~\ref{int:geodesic} gives the existence of
\be\label{def:lengthconstant}
\alpha:=\lim_{n\to\infty}\frac{\E\big[N(I,v_{n,i})\big]}{n},
\quad\text{and}\quad\sigma_N^2:=\lim_{n\to\infty}\frac{\Var\big(N(I,v_{n,i})\big)}{n}.
\ee

Geodesics are, as seen via Lemma \ref{lma:Ank}, locally determined. Thus, it makes sense to talk about an infinite geodesic from $-\infty$ to $\infty$. Let to this end $\gamma^\ast$ denote the unique infinite path that between level $\rho_{k-1}$ and $\rho_k$ coincides with $\gamma(\hat{v}_{\rho_{k-1}},\hat{v}_{\rho_k})$, for each $k\in\Z$ (when $\rho_k$ is suitably defined for $k<0$). The resulting infinite path is indeed a geodesic, i.e., any finite portion $\tilde{\gamma}^\ast$ of $\gamma^\ast$ with endpoints $u$ and $v$ satisfies $T(\tilde{\gamma}^\ast)=T(u,v)$. It is possible to characterize time and length constants in terms of this infinite geodesic.

\begin{prop}\label{prop:alphagamma}
\bea
\begin{aligned}
\alpha\;&=\;\sum_{v\in\Vb_{\G_0}}P(v\in\gamma^\ast)\;=\;\sum_{e\in\Eb_{\G_0}^\ast}P(e\in\gamma^\ast),\\
\mu\;&=\;\sum_{e\in\Eb_{\G_0}^\ast}\E[\tau_e\cdot\ind_{\{e\in\gamma^\ast\}}]\;=\;\sum_{e\in\Eb_{\G_0}^\ast}\E[\tau_e|e\in\gamma^\ast]P(e\in\gamma^\ast).
\end{aligned}
\eea
\end{prop}

\begin{proof}
We will deduce the characterisation of $\alpha$ in terms of vertices, and leave the remaining cases, which are deduced similarly, to the reader. Observe that
\bea
N(u,w)=\sum_{k\in\Z}\sum_{v\in\Vb_{\G_k}}\ind_{\{v\in\gamma(u,w)\}}-1,\quad\text{and}\quad\lim_{n\to\infty}\frac{\E[N(\hat{v}_{-n},\hat{v}_n)]}{2n}=\alpha.
\eea
Define $N^\ast:=\sum_{k=-n+\sqrt{n}}^{n-\sqrt{n}}\sum_{v\in\Vb_{\G_k}}\ind_{\{v\in\gamma^\ast\}}$. Clearly
$$
\frac{\E[N^\ast]}{2n}\;=\;\frac{2(n-\sqrt{n})}{2n}\sum_{v\in\Vb_{\G_0}}P(v\in\gamma^\ast)\;\raw\;\sum_{v\in\Vb_{\G_0}}P(v\in\gamma^\ast),\quad\text{as }n\raw\infty,
$$
so we are finished if we show that $\E\big[|N(\hat{v}_{-n},\hat{v}_n)-N^\ast|\big]/n\to0$, as $n\to\infty$. Let
$$
D_n:=\big\{A_k\cap A_{-k-2M}\text{ occurs for some }k\in[n-\sqrt{n},n-2M]\big\},
$$
where $A_k$ and $M$ are as defined in Section \ref{sec:patch}. Let $\kappa_n:=\min\{k\geq n:A_k\cap A_{-k-2M}\text{ occurs}\}$. Trivially, $|N(\hat{v}_{-n},\hat{v}_n)-N^\ast|\leq4|\Vb_{\G_0}|(\kappa_n+2M)$. On the event $D_n$ we have
\bea
\begin{split}
N(\hat{v}_{-n},\hat{v}_n)-N^\ast\;&=\;\sum_{\substack{k>n-\sqrt{n}\\ k<-n+\sqrt{n}}}\sum_{v\in\Vb_{\G_k}}\ind_{\{v\in\gamma(u,w)\}}-1\\
&\leq\; 2|\Vb_{\G_0}|(\kappa_n+2M-n+\sqrt{n}).
\end{split}
\eea
Since $\E[\kappa_n-n]$ is bounded, we easily realise that
\bea
\begin{split}
\E\big[|N(\hat{v}_{-n},\hat{v}_n)-N^\ast|\big]\;&=\;\E\big[|N(\hat{v}_{-n},\hat{v}_n)-N^\ast|(\ind_{D_n}+\ind_{D_n^c})\big]\\
&\leq\;2|\Vb_{\G_0}|\big(\E[\kappa_n-n]+2M+\sqrt{n}\big)+4|\Vb_{\G_0}|\big(\E[\kappa_n]+2M\big)P(D_n^c)\\
&\leq\;4|\Vb_{\G_0}|\big(C+\sqrt{n}+nP(D_n^c)\big)\;=\;o(n).
\end{split}
\eea
As mentioned, the remaining characterisations are deduced similarly.
\end{proof}

\cite{benkalsch03} posed the question whether for first-passage percolation on the $\Z^d$ lattice, $P\big(\mathbf{0}\in\gamma(-\mathbf{n},\mathbf{n})\big)\raw0$ as $n\raw\infty$? One may pose a corresponding question for first-passage percolation on essentially 1-dimensional graphs such as the $(K,d)$-tube: How does $P(v\in\gamma^\ast)$ behave as $K\to\infty$?

Let the \emph{$(K,d)$-cylinder} be the graph obtained from the $(K,d)$-tube by connecting opposing vertices on the boundary. Alternatively, the graph can be described as the $\Z\times\Z_K^{d-1}$ nearest neighbour graph, where $\Z_K^{d-1}$ denotes the $(d-1)$-dimensional torus of width $K$. All vertices in this graph are equivalent, so Proposition~\ref{prop:alphagamma} says that every vertex in the $(K,d)$-cylinder satisfies
$$
P(v\in\gamma^\ast)\,=\,\frac{\alpha}{K^{d-1}},
$$
for some constant $\alpha=\alpha(K,d)$. An argument due to Kesten shows that there is a finite constant $C=C(d)$ such that $\alpha(K)\leq C$ for all $K\geq1$ (cf.\ \citet[page 146]{howard04}). Since also $\alpha\geq1$ by triviality, we conclude that for $(K,d)$-cylinders
$$
\frac{1}{K^{d-1}}\;\leq\; P(v\in\gamma^\ast)\;\leq\;\frac{C}{K^{d-1}},\quad\text{uniformly in $v$ and $K$}.
$$
We cannot rely on symmetry to deduce similar asymptotics for $(K,d)$-tubes (although Kesten's argument remains valid). Indeed $(K,d)$-tubes would be the more interesting case, and it even remains an open question whether $\max_{v\in\Vb}P(v\in\gamma^\ast)\raw0$ as $K\raw\infty$ in this case.

\subsection{Continuity of constants}

The following result is inspired by a similar result due to \cite{cox80} and \cite{coxkes81}, who in their case consider first-passage percolation on the $\Z^d$ lattice. Their proof of the lattice case is rather lengthy. Due to the regenerative behaviour in the case of essentially 1-dimensional periodic graphs, this case turns out to be much simpler.

\begin{prop}\label{prop:contofconst}
Let $F_m$ for $m=1,2,\ldots,\infty$ be distribution functions such that $F_m\stackrel{d}{\raw}F_\infty$ as $m\raw\infty$. Then, as $m\raw\infty$,
$$
\alpha(F_m)\raw\alpha(F_\infty)\quad\text{and}\quad\sigma_N(F_m)\raw\sigma_N(F_\infty).
$$
Assume further that there is a distribution function $V$ such that $F_m\geq V$ for all $m\geq1$. If $\mu(V)<\infty$ and $\sigma(V)<\infty$ respectively, then as $m\raw\infty$
$$
\mu(F_m)\raw\mu(F_\infty)\quad\text{and}\quad\sigma(F_m)\raw\sigma(F_\infty).
$$
\end{prop}

In order to compare different distributions, we will couple random variables via their inverse distribution functions $F^{-1}(u):=\inf\{x\in\R:F(x)\geq u\}$. (The same approach is used in \cite{cox80} and \cite{coxkes81}.) Indeed, if $U$ is uniformly distributed on $[0,1]$, then $F^{-1}(U)$ has distribution $F$.
Thus, when $F$ runs over the class of distribution functions, then $\{F^{-1}(U)\}_F$ generates a coupling of all differently distributed random variables.
It is not hard to prove that as $m\raw\infty$, $F_m\stackrel{d}{\raw}F_\infty$ implies $F_m^{-1}(U)\raw F_\infty^{-1}(U)$ almost surely (see e.g.\ \cite[Section 1.8.4]{thorisson00}).

Once we have the above coupling, the idea is the following. If $F_m\stackrel{d}{\raw}F_\infty$, $F_m\geq V$ for all $m$, and $V$ has finite mean, then $F_m^{-1}(U)\leq V^{-1}(U)$ and $\E\big[F_m^{-1}(U)\big]\raw\E\big[F_\infty^{-1}(U)\big]$ as $m\raw\infty$, by the Dominated Convergence Theorem. The regenerative structure allows a similar approach.

\begin{proof}[Proof of Proposition \ref{prop:contofconst}]
Let $\{U_e\}_{e\in\Eb}$ be a collection of independent random variables uniformly distributed on $[0,1]$. Thus, as $F$ ranges over the class of passage-time distributions, then $\big\{\{F^{-1}(U_e)\}_{e\in\Eb}\big\}_F$ simultaneously couples i.i.d.\ sets of passage times of the graph. Choose $a\in(0,1/2)$ such that $F_\infty^{-1}(1-a)>F_\infty^{-1}(a)$, and $F_\infty^{-1}$ is continuous in both $a$ and $1-a$. Take $\epsilon>0$ such that $F_\infty^{-1}(1-a)-F_\infty^{-1}(a)>2\epsilon$. Choose $L<\infty$ such that
$$
\big|F_m^{-1}(a)-F_\infty^{-1}(a)\big|\leq\epsilon\quad\text{and}\quad\big|F_m^{-1}(1-a)-F_\infty^{-1}(1-a)\big|\leq\epsilon,
$$
for all $m\geq L$. Recall the definition of $A_n=A_n(M,\tlow,\thi)$ in Section \ref{sec:patch}. Set $\tlow=F_\infty^{-1}(a)+\epsilon$ and $\thi=F_\infty^{-1}(1-a)-\epsilon$, and let $M$ be chosen in accordance with Lemma \ref{lma:Ank}. For the same $M$ (and with notation as in Section \ref{sec:patch}), define
$$
\tilde{A}_n=\tilde{A}_n(M):=\big\{U_e\leq a, \forall e\in\hat{E}_n\big\}\cap\big\{U_e\geq1-a, \forall e\in E_n\setminus\hat{E}_n\big\}.
$$
Since $a>0$, we have $P(\tilde{A}_n)>0$. For all $m\geq L$ we have
\bea
\left\{
\begin{aligned}
F_m^{-1}(u)&\leq\tlow,\quad\text{for }u\leq a,\\
F_m^{-1}(u)&\geq\thi,\quad\text{for }u\geq1-a.
\end{aligned}
\right.
\eea
With a slight abuse of notation, we let $A_n(F)$ denote the event $A_n$ with respect to $\big\{F^{-1}(U_e)\big\}_{e\in\Eb}$. In particular, this implies that $\tilde{A}_n\subseteq A_n(F_m)$ for all $L\leq m\geq\infty$. Define a sequence $\{\tilde{\rho}_k\}_{k\geq0}$ with respect to $\tilde{A}_n$ analogously as in Section \ref{sec:patch}. Note that for $m\geq L$, the sequence $\{\tilde{\rho}_k\}_{k\geq0}$ is a subsequence of $\{\rho_k\}_{k\geq0}$ defined with respect to $A_n(F_m)$. The advantage of this is that we get a regenerative sequence valid for all distributions $F_m$ with $L\leq m\leq\infty$.

From here the result follows quickly. Let $T_F(u,v)$ denote the passage time between $u$ and $v$ with respect to $\big\{F^{-1}(U_e)\big\}_{e\in\Eb}$. For $m\geq L$ we have the characterisation
\bea
\mu(F_m)=\frac{\E\big[T_{F_m}(\hat{v}_{\tilde{\rho}_0},\hat{v}_{\tilde{\rho}_1})\big]}{\E[\tilde{\rho}_1-\tilde{\rho}_0]}.
\eea
Thus, in order to prove that $\mu(F_m)\raw\mu(F_\infty)$ as $m\raw\infty$, it suffices to show that
\be\label{eq:ETFm}
\E\big[T_{F_m}(\hat{v}_{\tilde{\rho}_0},\hat{v}_{\tilde{\rho}_1})\big]\raw\E\big[T_{F_\infty}(\hat{v}_{\tilde{\rho}_0},\hat{v}_{\tilde{\rho}_1})\big],\quad\text{as }m\raw\infty.
\ee
But, as $m\to\infty$, $F_m^{-1}(U)\raw F_\infty^{-1}(U)$, and hence $T_{F_m}(\hat{v}_{\tilde{\rho}_0},\hat{v}_{\tilde{\rho}_1})\raw T_{F_\infty}(\hat{v}_{\tilde{\rho}_0},\hat{v}_{\tilde{\rho}_1})$, almost surely. Since $T_{F_m}(\hat{v}_{\tilde{\rho}_0},\hat{v}_{\tilde{\rho}_1})\leq T_V(\hat{v}_{\tilde{\rho}_0},\hat{v}_{\tilde{\rho}_1})$ and the latter has finite mean, we conclude by the Dominated Convergence Theorem that (\ref{eq:ETFm}) holds. The remaining conclusions are drawn similarly.
\end{proof}

\begin{remark}
The true condition for the convergence $\E[F_m^{-1}(U)]\raw\E[F_\infty^{-1}(U)]$ is uniform integrability of $\{F_m\}_{m\geq1}$. In the same way it is possible to relax the condition in the above proposition. Assume that there are $p$ (edge) disjoint paths from $\hat{v}_0$ and $\hat{v}_1$, and let $Y_p(F)$ denote the minimum of $p$ independent random variables distributed as $F$. The precise condition for convergence of $\mu(F_m)$ and $\sigma(F_m)$ is uniform integrability of $\{Y_p(F_m)^\alpha\}_{m\geq1}$, for $\alpha=1$ and 2 respectively.
\end{remark}

\section{Exact coupling and a 0--1 law}\label{sec:coupling}

A \emph{coupling}\nocite{lindvall02} of two random variables $X\sim P$ and $Y\sim P^\prime$ on a measurable space $(E,\mathcal{E})$, is a joint distribution $\hat{P}$ of $(X,Y)$, i.e., a measure on $(E^2,\mathcal{E}^2)$, such that its marginal distributions coincide with $P$ and $P^\prime$. When we couple two time-dependent random elements $\{X_t\}_{t\geq0}$ and $\{Y_t\}_{t\geq0}$, we say that the coupling is \emph{exact} if with probability one there exists a $\Tc<\infty$ such that $X_t=Y_t$, for all $t\geq\Tc$.

We will present an exact coupling of the sets of infected vertices $B_t$ and $B_t^\prime$ of two first-passage percolation processes with different initial configurations. Recall that we let $\Pt(\;\cdot\;)$ denote the distribution of $\tau_e$, and let $\mathcal{R}_+$ denote the Borel $\sigma$-algebra on $[0,\infty)$. Then $\{\tau_e\}_{e\in\Eb}$ and $\{\tau_e^\prime\}_{e\in\Eb}$ are random elements on the product space $\left([0,\infty)^{\Eb},\mathcal{R}_+^{\Eb}\right)$, each with distribution given by the product measure $\Pt^{\Eb}$. Let $\Eb_n$ denote the set of edges between level $-n$ and $n$, but not including edges between two vertices at level $-n$ and $n$. In the same manner $\Eb_n^c$ denotes the set of edges at and before level $-n$, as well as at level $n$ and beyond.

\begin{prop}[Coupling, continuous times]\label{prop:coupling:c}
Let $I$ and $I^\prime$ be finite subsets of the set of vertices $\Vb$ of an essentially 1-dimensional periodic graph $\G$. Assume that the passage time distribution $\Pt$ has an absolutely continuous component (with respect to Lebesgue measure). For any $m\geq0$, there exists a coupling of $\{\tau_e\}_{e\in\Eb_m^c}$ and $\{\tau_e^\prime\}_{e\in\Eb_m^c}$ such that if $\{\tau_e\}_{e\in\Eb_m}$ and $\{\tau_e^\prime\}_{e\in\Eb_m}$ each have distribution $\Pt^{\Eb_m}$, then the marginal distributions of $\{\tau_e\}_{e\in\Eb}$ and $\{\tau_e^\prime\}_{e\in\Eb}$ are given by the product measure $\Pt^\Eb$, and such that if first-passage percolation is performed with $\big(I,\{\tau_e\}_{e\in\Eb}\big)$ and $\big(I^\prime,\{\tau_e^\prime\}_{e\in\Eb}\big)$, respectively, then with probability one there exists an $\Nc<\infty$ such that
\be\label{eq:coupling:c}
T(I,v_{n,i})=\Tp(I',v_{n,i})\quad\text{and}\quad B_t=B_t^\prime,
\ee
for all $i$, all $n\geq\Nc$, and all $t\geq\Nc$.
\end{prop}

When the passage time distribution $\Pt$ is discrete, i.e., $\Pt(\Lambda)=1$ for the set of point masses
$$
\Lambda:=\{t_j\in[0,\infty):\Pt(t_j)>0\},
$$
the statement of Proposition \ref{prop:coupling:c} is not true in general (cf.\ Remark \ref{rem:no1dcoupling}). In the discrete case, we will therefore restrict our attention to the case of $(K,d)$-tubes.

\begin{prop}[Coupling, discrete times]\label{prop:coupling:d}
Let $I$ and $I^\prime$ be finite subsets of the set of vertices $\Vb$ of the $(K,d)$-tube, for $K,d\geq2$. Assume that the passage time distribution $\Pt$ is such that $\Pt(\Lambda)=1$ for the set of point masses $\Lambda$ and that either of the following hold:
\begin{enumerate}[\quad a)]
\item\label{eq:dPtcond} there are $t_j\in\Lambda$ and integers $n_j$ for $j$ in some finite set of indices $J^\ast$, such that
$$\sum_{j\in J^\ast}n_j\text{ is odd},\quad\text{and}\quad\sum_{j\in J^\ast}n_jt_j=0.
$$
\item\label{eq:distcond} $\dist(\mathbf{x},\mathbf{y})$ is even, for all $\mathbf{x}\in I$, $\mathbf{y}\in I^\prime$.
\end{enumerate}
For any $m\geq0$, there exists a coupling of $\{\tau_e\}_{e\in\Eb_m^c}$ and $\{\tau_e^\prime\}_{e\in\Eb_m^c}$ such that if $\{\tau_e\}_{e\in\Eb_m}$ and $\{\tau_e^\prime\}_{e\in\Eb_m}$ each have distribution $\Pt^{\Eb_m}$, then the marginal distributions of $\{\tau_e\}_{e\in\Eb}$ and $\{\tau_e^\prime\}_{e\in\Eb}$ are given by the product measure $\Pt^\Eb$, and such that if first-passage percolation is performed with $\big(I,\{\tau_e\}_{e\in\Eb}\big)$ and $\big(I^\prime,\{\tau_e^\prime\}_{e\in\Eb}\big)$, respectively, then with probability one there exists an $\Nc<\infty$ such that
\be\label{eq:coupling:d}
T(I,v_{n,i})=\Tp(I',v_{n,i})\quad\text{and}\quad B_t=B_t^\prime,
\ee
for all $i$, all $n\geq\Nc$, and all $t\geq\Nc$.
\end{prop}

Before we construct the couplings, we focus on the promised 0--1 law that follows from Proposition \ref{prop:coupling:c} and \ref{prop:coupling:d}. For this we will use \emph{L\'evy's 0--1 law}. It states that for $\sigma$-algebras $\{\F_t\}_{t\geq0}$ such that $\F_t\uparrow\F_\infty$ as $t\raw\infty$, if $A\in\F_\infty$, then $P(A|\F_t)\raw\ind_A$, as $n\raw\infty$, almost surely. A proof for the discrete case can be found in e.g.\ \citet[Theorem 4.5.8]{durrett05}. The continuous case follows via the Martingale Convergence Theorem.

Recall that $\T_t=\sigma\big(\{B_s\}_{s\geq t}\big)$, $\T=\bigcap_{t\geq0}\T_t$, and let $\F_t:=\sigma\big(\{B_s\}_{0\leq s\leq t}\big)$, As before $B_s$ is the set of infected vertices at time $s$, and we may think of $\T_t$ as the $\sigma$-algebra of events $A\in\sigma\big(\bigcup_{t\geq0}\F_t\big)$ that do not depend on the times at which vertices were infected before time $t$. Note that Theorem~\ref{int:0--1law} is a special case of Theorem~\ref{thm:0--1law}.

\begin{thm}[0--1 law]\label{thm:0--1law}
Consider first-passage percolation performed under the assumptions of either Proposition \ref{prop:coupling:c} or \ref{prop:coupling:d}. Then $P(A)\in\{0,1\}$, for any event $A\in\T$.
\end{thm}

\begin{proof}[Proof of Theorem \ref{thm:0--1law} from Propositions \ref{prop:coupling:c} and \ref{prop:coupling:d}]
Consider two infections with the respective sets of passage times $\{\tau_e\}_{e\in\Eb}$ and $\{\tau_e^\prime\}_{e\in\Eb}$. For $t\geq0$, let $\F_t$ and $\F_t^\prime$ be $\sigma$-algebras generated by their respective realisations up to time $t$. Let
$$
\nu_t=\max\big\{n\geq0:(B_t\cup B_t^\prime)\cap(\Vb_{\G_n}\cup\Vb_{\G_{-n}})\neq\emptyset\big\}
$$
denote the furthest level (in positive or negative direction) infected at time $t$. Clearly $\nu_t<\infty$ almost surely, for every $t<\infty$.

For any fixed $t\geq0$, by Propositions \ref{prop:coupling:c} and \ref{prop:coupling:d}, there is a coupling of $\{\tau_e\}_{e\in\Eb_{\nu_t+1}^c}$ and $\{\tau_e^\prime\}_{e\in\Eb_{\nu_t+1}^c}$, such that there exists an almost surely finite time $\Nc$, such that $B_s=B_s^\prime$ for all $s\geq\Nc$. Since $A\in\T_{\Nc}$, the outcome of $A$ only depends on $B_s$ for $s\geq\Nc$. In particular
$$
P(A|\F_t)=P(A|\F_t^\prime).
$$
Thus, $P(A|\F_t)$ is nonrandom and equals $P(A)$, for all $t\geq0$. But, according to L\'evy's 0--1 law, $P(A|\F_t)\raw\ind_A$ as $t\raw\infty$, almost surely. Hence, $P(A)=\ind_A$ almost surely, and $P(A)\in\{0,1\}$.
\end{proof}

\subsection{Exact coupling of time-delayed infections on $\Z$}\label{sec:dcoupling}

Before proving Proposition \ref{prop:coupling:c} and \ref{prop:coupling:d}, we shall first provide a coupling of two infections on $\Z$, where one is delayed for some time $\Td$. This lemma will figure as a key step in the proof of Proposition \ref{prop:coupling:c} and \ref{prop:coupling:d}.

\begin{lma}\label{lma:coupling}
Let $\Td$ be a non-negative constant, and assume that either of the following hold:
\begin{enumerate}[\quad a)]
\item\label{lma:coupling:c} $\Pt$ has an absolutely continuous component (with respect to Lebesgue measure).
\item\label{lma:coupling:d} $\Pt$ is such that for some finite index set $J$, there are non-negative integers $n_j$ and $\np_j$, such that $\sum_{j\in J}n_j=\sum_{j\in J}\np_j$, and for atoms $t_j\in\Lambda$ of $\Pt$
\be\label{eq:dPtdelay}
\sum_{j\in J}n_jt_j\;=\;\sum_{j\in J}\np_jt_j+\Td.
\ee
\end{enumerate}
Then, there exists a coupling of $\{\tau_k\}_{k\geq1}$ and $\{\tp_k\}_{k\geq1}$ such that their marginal distributions are that of i.i.d.\ random variables with distribution $\Pt$, and such that almost surely
\be\label{eq:Zcoupling}
\sum_{k=1}^n\tau_k\;=\;\Td+\sum_{k=1}^n\tp_k,\quad\text{for large }n.
\ee
\end{lma}

The key to prove this lemma is to identify suitable random walks. The identification of the random walk in case \emph{a)} heavily exploits ideas similar to those found in \citet[Chapter III.5]{lindvall02}. In case \emph{b)}, a multi dimensional random walk will be based on condition (\ref{eq:dPtdelay}). This walk is then coupled with techniques found e.g.\ in \citet[Chapter II.12--17]{lindvall02}.

\begin{proof}[Proof of case \ref{lma:coupling:c})]
Let $[a,b]$ be an interval on which $\Pt$ has density $\geq c$, for some $c>0$. Define
$$
\delta:=\max\Big\{d\geq0:d\leq\frac{b-a}{2},d=\frac{\Td}{m}\text{ for some }m\in\N\Big\},
$$
and couple $\{\tau_k\}_{k\geq1}$ and $\{\tau_k^\prime\}_{k\geq1}$ in the following way. With probability $1-c2\delta$ we choose $\tau_k=\tp_k$, drawn from the distribution
$$
\tilde{\Pt}(\;\cdot\;):=\big(\Pt(\;\cdot\;)-c\lambda(\;\cdot\cap[a,a+2\delta])\big)\big/(1-c2\delta),
$$
where $\lambda$ denotes Lebesgue measure. With the remaining probability $c2\delta$, draw $\tau_k$ uniformly on the interval $[a,a+2\delta]$, and choose $\tp_k$ as
\bea
\tp_k=
\left\{
\begin{aligned}
&\tau_k+\delta, & &\text{if }\tau_k\leq a+\delta,\\
&\tau_k-\delta, & &\text{if }\tau_k> a+\delta.
\end{aligned}
\right.
\eea
That $\tp_k$ also is uniformly distributed on $[a,a+2\delta]$ is immediate. Thus, it is easy to see that the marginal distribution of both $\tau_k$ and $\tp_k$ is $\Pt$, and this is indeed a coupling of the two infections.

\begin{figure}[htbp]
\begin{center}
\resizebox{0.9\textwidth}{!}{\input{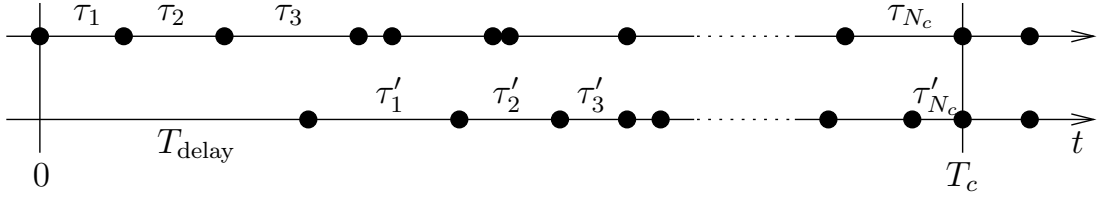}}
\end{center}
\caption{The dots represent the times at which the respective infection spreads. In this realisation $\tau_1=\tp_1-\delta$, $\tau_2=\tp_2$ and $\tau_3=\tp_3+\delta$.}
\label{fig:DnonZdot}
\end{figure}
The coupling is constructed so that each time $\tau_k$ and $\tp_k$ are chosen differently, the difference $D_n:=\Td+\sum_{k=1}^n(\tp_k-\tau_k)$ will jump $\pm\delta$. Since $\Td=m\delta$, for some integer $m$, $\{D_n\}_{n\geq1}$ constitutes a simple random walk on $\delta\Z$. Let $\Nc$ denote the first $n$ for which $D_n$ hits zero. From this moment on, $\tau_k$ and $\tp_k$ are chosen identically, and (\ref{eq:Zcoupling}) holds for $n\geq\Nc$. That the coupling is successful follows from the recurrence of 1-dimensional (lazy) simple random walks.\\

\noindent
{\it Proof of case \ref{lma:coupling:d})}.
By assumption, for some set $\{t_j\}_{j\in J}\subseteq\Lambda$ of atoms for the distribution $\Pt$, there are non-negative integers $n_j$ and $\np_j$ such that $\sum_{j\in J}n_j=\sum_{j\in J}\np_j$ and (\ref{eq:dPtdelay}) holds. It is easily seen that we may assume that $J$, $n_j$ and $\np_j$ are chosen such that for each $j\in J$, exactly one of the integers $n_j$ and $\np_j$ is positive. We introduce integer valued random variables
\bea
\begin{aligned}
X_j^n\;&=\;\#\{k\leq n:\tau_k=t_j\}-n_j,\\
Y_j^n\;&=\;\#\{k\leq n:\tp_k=t_j\}-\np_j.
\end{aligned}
\eea
Define $Z_j^n=X_j^n-Y_j^n$. It is clear that we from (\ref{eq:dPtdelay}) can conclude that (\ref{eq:Zcoupling}) holds, if $Z_j^n=0$ for all $j\in J$ and $\tau_k=\tp_k$ for all $k\leq n$ such that $\tau_k\not\in\{t_j\}_{j\in J}$ or $\tp_k\not\in\{t_j\}_{j\in J}$.

Let $J_n=\{j\in J:Z_j^n\neq0\}$, let $p_j=\Pt(t_j)$, and $q_n=\sum_{j\in J_n}p_j$. In particular, $J_0=J$. Couple $\{\tau_k\}_{k\geq1}$ and $\{\tp_k\}_{k\geq1}$ by choosing $\tau_k$ and $\tp_k$ identically from the distribution
$$
\tilde{\Pt}(\;\cdot\;):=\frac{1}{1-q_{k-1}}\bigg(\Pt(\;\cdot\;)-\sum_{j\in J_{k-1}}p_j\mathbf{1}_{\{t_j\}}(\;\cdot\;)\bigg)
$$
with probability $1-q_{k-1}$. With remaining probability $q_{k-1}$ we choose $\tau_k$ and $\tp_k$ independently with distribution $P(\tau=t_j)=p_j/q_{k-1}$, for $j\in J_{k-1}$. The marginal distribution of $\tau_k$ and $\tp_k$ is readily seen to be $\Pt$, whence this is a coupling of $\{\tau_k\}_{k\geq1}$ and $\{\tp_k\}_{k\geq1}$.

Note that $\tau_k=\tp_k$ for all $k$ such that $\tau_k\not\in\{t_j\}_{j\in J}$ and $\tp_k\not\in\{t_j\}_{j\in J}$. For each fixed $j\in J$, $\{Z_j^n\}_{n\geq0}$ will, as $n$ increases, jump $\pm1$ with equal probability. Hence, for fixed $j$, $\{Z_j^n\}_{n\geq0}$ constitutes a (lazy) simple random walk on $\Z$. Note that if $n^\ast$ denotes the first $n$ such that $Z_j^n=0$, then, by definition, $j\in J_n$ for $n<n^\ast$, but $j\not\in J_n$ for $n\geq n^\ast$.

By assumption we have that
$$
0\;=\;\sum_{j\in J}\left(n_j-\np_j\right)\;=\;\sum_{j\in J}Z_j^0\;=\;\sum_{j\in J}Z_j^n,\quad\text{for all }n\geq0.
$$
It follows that $|J_n|\neq1$ for all $n$. There will therefore always be a positive probability to choose $\tau_{n+1}\neq\tau_{n+1}^\prime$ as long as $Z_j^n\neq0$ for some $j$. From this observation and the recurrence of 1-dimensional simple random walks, we conclude that $N_c=\min\{n\geq0: Z_j^n=0\text{ for all }j\in J\}$ is almost surely finite.
\end{proof}

\subsection{Exact coupling of two infections}

In order to prove Proposition \ref{prop:coupling:c} and \ref{prop:coupling:d}, we will arrange matters so that Lemma \ref{lma:coupling} can be applied. We first outline the general idea. It follows from the regenerative behaviour that if $\tau_e=\tau_e^\prime$ for all $e\in\Eb$, then there is a real number $T_d$ such that
\bea
T(I,v_{n,i})-T'(I',v_{n,i})\,=\,T_d
\eea
for all $i$ and all $n$ large enough. The idea will be to assign identical passage times for both infections, that is $\tau_e=\tau_e^\prime$, except for certain edges which we make sure both infections have to pass. This generates a sequence of edges for which we invoke Lemma \ref{lma:coupling}. This will complete the coupling for positive levels $n$. The opposite direction is treated analogously.

To make this precise, recall the notation in (\ref{eq:hEn}) and (\ref{eq:Ank}). Introduce the notation $\hat{e}_{n+M}$ for the edge in $\gamma_n$ with endpoints $\hat{v}_{n+M}$ and $u$, where $\hat{v}_{n+M}$ is the vertex in $\Vb_{\G_{n+M}}$ first reached by $\gamma_{n}$, and $u$ the vertex first reached after $\hat{v}_{n+M}$ by $\gamma_{n}$. Define the event
$$
A_n^\ast:=\big\{\tau_e\leq\tlow, \forall e\in\hat{E}_n\setminus\{\hat{e}_{n+M}\}\big\}\cap\big\{\tau_e\geq\thi, \forall e\in E_n\setminus\hat{E}_n\big\}.
$$
Note that $A_n=A_n^\ast\cap\{\tau_{\hat{e}_{n+M}}\leq\tlow\}$ for $A_n$ as defined in (\ref{eq:Ank}).

\begin{proof}[{\bf Proof of Proposition \ref{prop:coupling:c}}]
By assumption, $\Pt$ has an absolutely continuous component, so suppose that $[a,b]$ is an interval on which $\Pt$ has density $\geq c>0$. Let $a<\tlow<\thi<b$ and choose $M$ in accordance with Lemma \ref{lma:Ank}. We may further assume that $I\cup I^\prime$ contains no vertex beyond level $m$. Let $l_k:=m+k(2M+1)$ for $k\geq0$. Couple $\{\tau_e\}_{e\in\Eb_m^c}$ and $\{\tau_e^\prime\}_{e\in\Eb_m^c}$ by choosing $\tau_e=\tp_e$ with distribution $\Pt$, independently for all $e$ at level $m$ or beyond such that $e\neq\hat{e}_{l_k+M}$ for some $k\geq0$. Independently for $k\geq0$, let
\bea
(\xi_k,\xi_k^\prime)=
\left\{
\begin{aligned}
&(\theta_k,\theta_k^\prime), & &\text{with probability }\Pt([0,\tlow]),\\
&(\eta_k,\eta_k), & &\text{with probability }1-\Pt([0,\tlow]),
\end{aligned}
\right.
\eea
where $\theta_k$ and $\theta_k^\prime$ are to be coupled below, so that they both have marginal distribution $\Pt(\;\cdot\;|\tau\leq\tlow)$, and $\eta_k$ has distribution $\Pt(\;\cdot\;|\tau>\tlow)$. For the set of edges $\{\hat{e}_{l_k+M},\text{ for }k\geq0\}$, we choose the pair
\bea
\left(\tau_{\hat{e}_{l_k+M}},\tp_{\hat{e}_{l_k+M}}\right)=
\left\{
\begin{aligned}
&(\xi_k,\xi_k^\prime), & &\text{if }A_{l_k}^\ast\text{ occurs}\\
&(\tau_k,\tau_k), & &\text{otherwise},
\end{aligned}
\right.
\eea
where $\tau_k$ is distributed according to $\Pt$, independently for all $k$. One realises from the coupling that the marginal distributions of both $\tau_e$ and $\tp_e$ is $\Pt$, for every edge $e$.

Note that the only edges for which $\tau_e$ and $\tp_e$ may differ, are the edges $\hat{e}_{l_k+M}$ for $k\geq0$ such that $A_{l_k}$ occurs. Let $\kappa_j$ denote the index $k$ for which $A_{l_k}$ occurs for the $j$th time. That
\be\label{eq:couplingtaurho}
\left(\tau_{\hat{e}_{l_k+M}},\tp_{\hat{e}_{l_k+M}}\right)=(\theta_k,\theta_k^\prime)
\ee
is equivalent to that $A_{l_k}$ occurs. Since $P(A_{l_k})>0$, we will have an infinite sequence $\{\kappa_j\}_{j\geq1}$ such that (\ref{eq:couplingtaurho}) holds. Via Lemma~\ref{lma:Ank} we conclude that $\hat{e}_{l_{\kappa_j}}$ has to be passed by both infections, and that
$$
T(I,\hat{v}_{l_{\kappa_j}+M})-T'(I',\hat{v}_{l_{\kappa_j}+M})=T(I,\hat{v}_{l_{\kappa_1}+M})-T'(I',\hat{v}_{l_{\kappa_1}+M})+\sum_{i=1}^{j-1}\theta_{\kappa_i}-\theta'_{\kappa_i}.
$$
Apply Lemma \ref{lma:coupling} to $\{\theta_{\kappa_j}\}_{j\geq1}$ and $\{\theta_{\kappa_j}^\prime\}_{j\geq1}$, with distribution $\Pt(\;\cdot\;|\tau\leq\tlow)$, and $\Td=\big|T(I,\hat{v}_{l_{\kappa_1}+M})-\Tp(I',\hat{v}_{l_{\kappa_1}+M})\big|$. Since $\Pt(\;\cdot\;|\tau\leq\tlow)$ is absolutely continuous on $[a,\tlow]$, it follows that $T(I,v_{n,i})=T'(I',v_{n,i})$ for all $i$, and all $n$ large enough. Both infections can be coupled analogously in the negative direction, which would complete the construction.
\end{proof}

The proof of Proposition~\ref{prop:coupling:d} is a bit more involved than that of Proposition~\ref{prop:coupling:c}, since before applying Lemma~\ref{lma:coupling} we need to make sure that the geodesics attain the `correct' length. We outline the proof here, and refer the reader to~\cite{A11thesis} for the remaining details.

Note that in the coupling constructed above, the difference between $N(I,v_{n,i})$ and $N'(I',v_{n,i})$ is constant for all $i$ and all but finitely many $n\geq0$. In order to succeed with the coupling in the discrete case, we first need to couple the infections so that this difference is either zero or the odd number $n^\ast=\sum_{j\in J^\ast}n_j$ figuring in the assumption of Proposition~\ref{prop:coupling:d}. This can be accomplished as follows. First, assign passage times equal for both infections. Second, define two events $C_n$ and $D_n$ such that if one occurs for either infection, then $N(I,v_{n,i})-N'(I',v_{n,i})$ changes by 2 for all $i$ and all but finitely many $n\geq0$ (cf.\ Figure~\ref{fig:3tubeDn}).
\begin{figure}[htbp]
\begin{center}
\resizebox{0.9\textwidth}{!}{\input{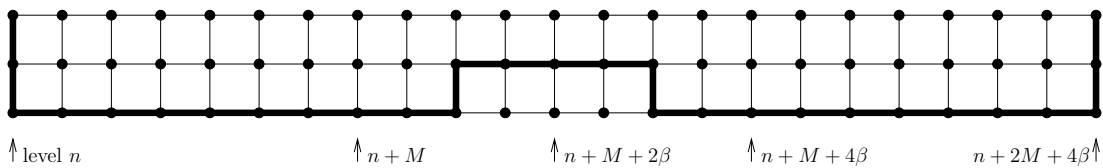}}
\end{center}
\caption{If $D_n$ occurs, then the infection is forced to follow the thick line, whereas if $C_n$ had occurred, then the infection would have followed the straight line segment.}
\label{fig:3tubeDn}
\end{figure}
Repeating this procedure, we cause the difference to perform a random walk on either $2\Z$ or $2\Z+1$. If $\dist(x,y)$ is even for all $x\in I$ and $y\in I'$, then the walk liven on $2\Z$ and applying Lemma~\ref{lma:coupling} will be easy. If this is not the case, then the random walk may live on $2\Z+1$, and the additional condition \emph{b)} is necessary.

Once the difference between geodesics have attained the right value, the coupling will continue along the lines of the continuous case. However, we need to consider a variant of the event $A_n^\ast$, since the infections may have passed an edge with value $\Mt$ in order to reach this stage. The remaining details are presented where indicated above.

\begin{remark}
If $\dist(\mathbf{x},\mathbf{y})$ is odd, for all $\mathbf{x}\in I$, $\mathbf{y}\in I^\prime$, then condition \emph{a)} of Proposition \ref{prop:coupling:d} is necessary. To see this, assume that an exact coupling is possible. In particular, $T(I,v)=\Tp(I',v)$ for some vertex $v$. But, if one infection has an even number of edges to pass in order to reach $v$, the other has an odd number of edges to pass. Thus,
$$
0\;=\;T(I,v)-\Tp(I',v)\;=\;\sum_{j\in J}n_jt_j-\sum_{j\in J}\np_jt_j,
$$
for integers $n_j$ and $\np_j$ such that $\sum_{j\in J}(n_j-\np_j)$ is odd. Hence, condition \emph{a)} holds.
\end{remark}

\begin{remark}\label{rem:no1dcoupling}
Condition \emph{a)} of Proposition~\ref{prop:coupling:d} is not sufficient for the existence of an exact coupling on arbitrary essentially 1-dimensional periodic graphs. The distribution $\Pt(1)=\Pt(1+3/5)=1/2$ satisfies the condition, but it is not always possible to exactly couple two infections on the graph with vertex set $\Z\times\{0,1\}$ and two vertices are connected if their Euclidean distance is $\leq\sqrt{2}$.
However, both condition \emph{a)} and \emph{b)} of Proposition \ref{prop:coupling:d} could be dropped for e.g.\ the class of triangular graphs with vertex set $\Z\times\{0,1,\ldots,K-1\}$ and where two vertices at Euclidean distance is 1 and every two vertices $(n,m)$ and $(n+1,m+1)$ for any $n\in\Z$ and $m=0,1,\ldots,K-2$, are connected by an edge.
\end{remark}

\begin{remark}
Condition \emph{a)} of Lemma \ref{lma:coupling} can be weakened to distributions $\Pt$ whose convolution with itself has an absolutely continuous component. In fact, it is sufficient if $\Pt$ convoluted with itself $n$ times, for some $n\geq0$, has an absolutely continuous component. Since the distribution of a sum of independent random variables is the convolution of the individual distributions, we may instead of specifying how to choose $(\tau_j,\tp_j)$ for $j\geq1$, choose $\big(\sum_{k=(j-1)n+1}^{jn}\tau_k,\sum_{k=(j-1)n+1}^{jn}\tp_k\big)$ according to the same specification. Consequently, the assumption of Proposition \ref{prop:coupling:c} can be weakened correspondingly.
\end{remark}

\subsection{No exact coupling possible on trees}\label{sec:notreecoupling}

We have seen that there is an exact coupling of two first-passage percolation infections on any essentially 1-dimensional periodic graph when the passage time distribution has an absolutely continuous component. We also saw how this sort of coupling gave rise to a 0--1 law. One may ask whether a continuous component is sufficient for an analogous coupling, and corresponding 0--1 law, on any graph? We will answer this question no, by showing that the binary tree $\mathbb{T}^2$ constitutes a counterexample. $\mathbb{T}^2$ is the infinite graph that does not contain any circuit, and where each vertex has three neighbours. The graph is completely homogeneous and one vertex, called the \emph{root}, is chosen for reference. Let $\{\tau_e\}_{e\in\Eb}$ be a set of independent and exponentially distributed passage times associated with the edge set $\Eb$ of $\mathbb{T}^2$, and analogous to before, let
$$
B_t=\big\{v\in\Vb:T\left(\text{root},v\right)\leq t\big\}.
$$

The following argument is based on the theory of continuous branching processes\nocite{athney72}. Define the front line of the infection at time $t$ as
$$
F_t:=\#\big\{v\not\in B_t:v\text{ shares an edge with some }u\in B_t\big\}.
$$
Note that $F_0=3$ and that $F_t$ increases by one, when $B_t$ does. Hence, $F_t$ can be seen as a continuous time branching process with $F_t$ individuals at time $t$. Each individual gives with probability one birth to two children (and dies) after an exponentially distributed time, independent of one another. It is well-known (see e.g.\ \citet[Theorems III.7.1--2]{athney72}) that, for some Malthusian parameter $\lambda>0$,
\be\label{eq:existsW}
\exists W:=\lim_{t\raw\infty}F_te^{-\lambda t},\quad\text{almost surely},
\ee
and that $\E[W]=3$. Let $\tau_{e_1}$, $\tau_{e_2}$ and $\tau_{e_3}$ denote the passage time of the edges connected to the root, and let $\tilde{F}_t$ denote $F_t$ conditioned on $\{\tau_{e_1},\tau_{e_2},\tau_{e_3}\geq1\}$. Then, by the lack-of-memory property of the exponential distribution, we have that $\tilde{F}_{t+1}\stackrel{d}{=}F_t$ for any $t\geq0$. Thus, by (\ref{eq:existsW}) we have almost surely
$$
\lim_{t\raw\infty}\tilde{F}_te^{-\lambda t}\;\stackrel{d}{=}\;e^{-\lambda}\lim_{t\raw\infty}F_te^{-\lambda t}\;=\;e^{-\lambda}W,
$$
and we conclude that $W$ is almost surely non-constant. Note that the event
$$
\big\{W=\lim_{t\raw\infty}F_te^{-\lambda t}\leq x\big\}\in\T,\quad\text{for every }x.
$$
Then, a 0--1 law analogous to Theorem \ref{thm:0--1law} cannot hold for first-passage percolation on $\mathbb{T}^2$, since this would imply that $P(W\leq x)\in\{0,1\}$, i.e., that $W$ is almost surely constant.

\begin{acknow}
The author is grateful to Olle Häggström for introducing him to this problem, as well as for his valuable advice along the way. He would also like to thank Vladas Sidoravicius for a helpful discussion, Erik Broman for his constructive remarks on the manuscript, as well as Andreas Nordvall-Lagerås for pointing out the book by \cite{gut09}, which enriched an earlier version of this paper.
\end{acknow}

\bibliographystyle{plainnat}
\bibliography{bibpercolation}

\newcommand{\noopsort}[1]{}
\begin{thebibliography}{35}
\providecommand{\natexlab}[1]{#1}
\providecommand{\url}[1]{\texttt{#1}}
\expandafter\ifx\csname urlstyle\endcsname\relax
  \providecommand{\doi}[1]{doi: #1}\else
  \providecommand{\doi}{doi: \begingroup \urlstyle{rm}\Url}\fi

\bibitem[Ahlberg(2008)]{A08}
D.~Ahlberg.
\newblock Asymptotics of first-passage percolation on 1-dimensional graphs.
\newblock Licentiate thesis, available at
  \url{http://www.math.chalmers.se/Math/Research/Preprints/2008/39.pdf}, 2008.

\bibitem[Ahlberg(2011)]{A11thesis}
D.~Ahlberg.
\newblock Asymptotics and dynamics in first-passage and continuum percolation.
\newblock Ph.D.\ thesis, available at
  \url{http://gupea.ub.gu.se/handle/2077/26666}, 2011.

\bibitem[Ahlberg(2013{\natexlab{a}})]{A13}
D.~Ahlberg.
\newblock A {H}su-{R}obbins-{E}rd{\H o}s strong law in first-passage
  percolation.
\newblock Available as \emph{arXiv: 1305.6260}, 2013{\natexlab{a}}.

\bibitem[Ahlberg(2013{\natexlab{b}})]{A13-2}
D.~Ahlberg.
\newblock Convergence towards an asymptotic shape in first-passage percolation
  on cone-like subgraphs of the integer lattice.
\newblock Available as \emph{arXiv: 1107.2280v3}, 2013{\natexlab{b}}.

\bibitem[Athreya and Ney(1972)]{athney72}
K.~B. Athreya and P.E. Ney.
\newblock \emph{Branching processes}.
\newblock Springer, Berlin, 1972.

\bibitem[Bena{\"i}m and Rossignol(2006)]{benros06}
M.~Bena{\"i}m and R.~Rossignol.
\newblock A modified {P}oincar{\'e} inequality and its application to first
  passage percolation.
\newblock Available as \emph{arXiv: 0602496}, 2006.

\bibitem[Benjamini et~al.(2003)Benjamini, Kalai, and Schramm]{benkalsch03}
I.~Benjamini, G.~Kalai, and O.~Schramm.
\newblock First passage percolation has sublinear distance variance.
\newblock \emph{Ann. Probab.}, 31:\penalty0 1970--1978, 2003.

\bibitem[{\noopsort{Berg}}van~den Berg(1983)]{vandenberg83}
J.~{\noopsort{Berg}}van~den Berg.
\newblock A counterexample to a conjecture of {J.\ M.\ Hammersley} and {D.\ J.\
  A.\ Welsh} concerning first-passage percolation.
\newblock \emph{Adv. in Appl. Probab.}, 15:\penalty0 465--467, 1983.

\bibitem[Chatterjee and Dey(2012)]{chadey09}
S.~Chatterjee and P.~Dey.
\newblock Central limit theorem for first-passage percolation time across thin
  cylinders.
\newblock To appear in \emph{Probab. Theory Related Fields}, 2012.

\bibitem[Cox(1980)]{cox80}
J.~T. Cox.
\newblock The time constant of first-passage percolation on the square lattice.
\newblock \emph{Adv. in Appl. Probab.}, 12:\penalty0 864--879, 1980.

\bibitem[Cox and Durrett(1981)]{coxdur81}
J.~T. Cox and R.~Durrett.
\newblock Some limit theorems for percolation processes with necessary and
  sufficient conditions.
\newblock \emph{Ann. Probab.}, 9:\penalty0 583--603, 1981.

\bibitem[Cox and Kesten(1981)]{coxkes81}
J.~T. Cox and H.~Kesten.
\newblock On the continuity of the time constant of first-passage percolation.
\newblock \emph{J. Appl. Probab.}, 18:\penalty0 809--819, 1981.

\bibitem[Damron et~al.(2013)Damron, Hanson, and Sosoe]{damhansos13}
M.~Damron, J.~Hanson, and P.~Sosoe.
\newblock Sublinear variance in first-passage percolation for general
  distributions.
\newblock Available as \emph{arXiv: 1306.1197}, 2013.

\bibitem[Durrett(2005)]{durrett05}
R.~Durrett.
\newblock \emph{Probability: {T}heory and examples}.
\newblock Brooks/Cole [Thomson Learning], Belmont, third edition, 2005.

\bibitem[Flaxman et~al.(2011)Flaxman, Gamarnik, and Sorkin]{flagamsor11}
A.~Flaxman, D.~Gamarnik, and G.~B. Sorkin.
\newblock First-passage percolation on a ladder graph, and the path cost in a
  {VCG} auction.
\newblock \emph{Random Structures Algorithms}, 38:\penalty0 350--364, 2011.

\bibitem[Garet and Marchand(2004)]{garmar04}
O.~Garet and R.~Marchand.
\newblock Asymptotic shape for the chemical distance and first-passage
  percolation on the infinite {B}ernoulli cluster.
\newblock \emph{ESAIM Probab. Stat.}, 8:\penalty0 169--199, 2004.

\bibitem[Gou{\'e}r{\'e}(2012)]{gouere12}
J.-B. Gou{\'e}r{\'e}.
\newblock Monotonicity in first-passage percolation.
\newblock Available as \emph{arXiv: 1202:2665}, 2012.

\bibitem[Grimmett and Kesten(2012)]{grikes12}
G.~Grimmett and H.~Kesten.
\newblock Peroclation since {S}aint-{F}lour.
\newblock Available as \emph{arXiv: 1207.0373}, 2012.

\bibitem[Gut(2009)]{gut09}
A.~Gut.
\newblock \emph{Stopped random walks. {L}imit theorems and applications}.
\newblock Springer, New York, second edition, 2009.

\bibitem[Hammersley and Welsh(1965)]{hamwel65}
J.~M. Hammersley and D.~J.~A. Welsh.
\newblock First-passage percolation, subadditive processes, stochastic networks
  and generalized renewal theory.
\newblock In J.~Neyman and L.~Le~Cam, editors, \emph{Bernoulli, Bayes, Laplace
  Anniversary Volume}, Proc. Internat. Res. Semin., Statist. Lab., Univ.
  California, Berkeley, Calif., pages 61--110. Springer-Verlag, New York, 1965.

\bibitem[Howard(2004)]{howard04}
C.~D. Howard.
\newblock Models of first-passage percolation.
\newblock In H.~Kesten, editor, \emph{Probability on discrete structures},
  volume 110 of \emph{Encyclopaedia Math. Sci.}, pages 125--173. Springer,
  Berlin, 2004.

\bibitem[Kesten(1986)]{kesten86}
H.~Kesten.
\newblock Aspects of first-passage percolation.
\newblock In \emph{{\'E}cole d'{\'E}t{\'e} de Probabilit{\'e}s de Saint Flour
  XIV - 1984}, volume 1180 of \emph{Lecture Notes in Math.}, pages 125--264.
  Springer, Berlin, 1986.

\bibitem[Kesten(1993)]{kesten93}
H.~Kesten.
\newblock On the speed of convergence in first-passage percolation.
\newblock \emph{Ann. Appl. Probab.}, 3:\penalty0 296--338, 1993.

\bibitem[Kesten and Zhang(1997)]{keszha97}
H.~Kesten and Y.~Zhang.
\newblock A central limit theorem for {"critical"} first-passage percolation in
  two dimensions.
\newblock \emph{Probab. Theory Related Fields}, 107:\penalty0 137--160, 1997.

\bibitem[Kingman(1968)]{kingman68}
J.~F.~C. Kingman.
\newblock The ergodic theory of subadditive stochastic processes.
\newblock \emph{J. Roy. Statist. Soc. Ser. B}, 30:\penalty0 499--510, 1968.

\bibitem[Lindvall(2002)]{lindvall02}
T.~Lindvall.
\newblock \emph{Lectures on the coupling method}.
\newblock Dover, Mineola, 2002.

\bibitem[Newman and Piza(1995)]{newpiz95}
C.~M. Newman and M.~S.~T. Piza.
\newblock Divergence of shape fluctuations in two dimensions.
\newblock \emph{Ann. Probab.}, 23:\penalty0 977--1005, 1995.

\bibitem[Pemantle and Peres(1994)]{pemper94}
R.~Pemantle and Y.~Peres.
\newblock Planar first-passage percolation times are not tight.
\newblock In \emph{Probability and phase transition}, pages 261--264. Kluwer,
  Dordrecht, 1994.

\bibitem[Renlund(2010)]{renlund10}
H.~Renlund.
\newblock First-passage percolation with exponential times on a ladder.
\newblock \emph{Combin. Probab. Comput.}, 19:\penalty0 593--601, 2010.

\bibitem[Renlund(2011)]{renlund11}
H.~Renlund.
\newblock First-passage percolation on ladder-like graphs with heterogeneous
  exponential times.
\newblock Available as \emph{arXiv: 1102.4744}, 2011.

\bibitem[Richardson(1973)]{richardson73}
D.~Richardson.
\newblock Random growth in a tessellation.
\newblock \emph{Proc. Cambridge Philos. Soc.}, 74:\penalty0 515--528, 1973.

\bibitem[Schlemm(2009)]{schlemm09}
E.~Schlemm.
\newblock First-passage percolation rates on width-two stretches with
  exponential link weights.
\newblock \emph{Electron. Commun. Probab.}, 14:\penalty0 424--434, 2009.

\bibitem[Schlemm(2011)]{schlemm11}
E.~Schlemm.
\newblock On the {M}arkov transition kernels for first-passage percolation on
  the ladder.
\newblock \emph{J. Appl. Probab.}, 48:\penalty0 366--388, 2011.

\bibitem[Thorisson(2000)]{thorisson00}
H.~Thorisson.
\newblock \emph{Coupling, stationarity, and regeneration}.
\newblock Springer-Verlag, New York, 2000.

\bibitem[Zhang and Zhang(1984)]{zhazha84}
Y.~Zhang and Y.~Zhang.
\newblock A limit theorem for {$N_{0n}/n$} in first-passage percolation.
\newblock \emph{Ann. Probab.}, 12:\penalty0 1068--1076, 1984.

\end{thebibliography}

\end{document}